\documentclass[12pt]{amsart}
\pdfoutput=1
\usepackage[left=1in,top=1in,right=1in,bottom=1in,letterpaper]{geometry}
\usepackage{epsfig}
\usepackage{amsmath}
\usepackage{amssymb}
\usepackage{amscd}
\usepackage{url}
\usepackage{verbatim} 
\usepackage{color}
\usepackage{epsfig}
\usepackage{amsthm}
\usepackage{pifont}

\usepackage{hyperref}

\title[An infinite surface with the lattice property]{An infinite surface with the lattice property I:\\
Veech groups and coding geodesics}
\author{W. Patrick Hooper}
\thanks{Supported by N.S.F. Postdoctoral Fellowship DMS-0803013, N.S.F. Grant DMS-1101233, and a PSC-CUNY Award (funded by The Professional Staff Congress and The City University of New York).}
\address{Department of Mathematics, City College of New York}
\email{whooper@ccny.cuny.edu}%\phone{847-491-2853}

\ifdefined\theorem
\else
\newtheorem{theorem}{Theorem}
\newtheorem{proposition}[theorem]{Proposition}
\newtheorem{lemma}[theorem]{Lemma}
\newtheorem{remark}[theorem]{Remark}
\newtheorem{corollary}[theorem]{Corollary}

\theoremstyle{definition}
\newtheorem{definition}[theorem]{Definition}
\newtheorem{notation}[theorem]{Notation}

\fi
\ifdefined\openquestion
\else
\newtheorem{openquestion}[theorem]{Open Question}
\fi

\ifdefined\question
\else
\newtheorem{question}[theorem]{Question}
\fi

\newlength{\savearraycolsep}
	{\setlength{\savearraycolsep}{\arraycolsep}%
	\setlength{\arraycolsep}{#1}%
	\begin{array}{#2}}%
	{\end{array}\setlength{\arraycolsep}{\savearraycolsep}}

\def\C{\mathbb{C}}% 
\def\D{\mathbb{D}}% 
\def\R{\mathbb{R}}% 
\def\Z{\mathbb{Z}}% 
\def\RP{\mathbb{RP}}%
 
\def\Circ{\mathbb{S}^1} % the circle
\def\H{\mathbb{H}} % Hyperbolic space
\def\Isom{{\mathit{Isom}}} % Isometry group

\def\til{\widetilde}

\def\sm{\smallsetminus}

\def\GL{\textit{GL}}
\def\SL{\textit{SL}}
\def\SO{\textit{SO}}
\def\PGL{\textit{PGL}}
\def\PSL{\textit{PSL}}

\def\0{{\mathbf{0}}}
\def\1{{\mathbf{1}}}

\def\u{{\mathbf{u}}}

\def\G{{\mathbf{G}}}

\def\sF{{\mathcal{F}}}

\def\sL{{\mathcal{L}}}
\def\sM{{\mathcal{M}}}

\def\sT{{\mathcal{T}}}

\newcommand{\nullset}{\emptyset}
\def\hol{\mathit{hol}} % Holonomy map
\makeatletter
\def\imod#1{\allowbreak\mkern10mu({\operator@font mod}\,\,#1)}
\makeatother

\def\dev{{\mathit{dev}}} % developing map 
\def\Aff{\mathit{Aff}} % Affine automorphism group

\def\and{{\quad \textrm{and} \quad}}

\def\G{{\mathcal G}}% Affine Automorphism Group
\def\dev{{\mathit{dev}}} % Developing Map
\def\D{{\mathbf{D}}} % derivative map

\begin{document}
\begin{abstract}
We study the symmetries and geodesics of an infinite translation surface which arises
as a limit of translation surfaces built from regular polygons, studied by Veech.
We find the affine symmetry group of this infinite translation surface, and we show
that this surface admits a deformation into other surfaces with topologically equivalent affine symmetries. 
The geodesics on these new surfaces are combinatorially the same as the geodesics on the original.
\end{abstract}

\maketitle

In this paper, we begin a systematic study of the geometric and dynamical properties of the surface $S_1$ shown below.
This surface arises from a limit of surfaces built from two affinely regular $n$-gons as $n \to \infty$. 

\begin{figure}[ht]
\begin{center}
\includegraphics[width=3in]{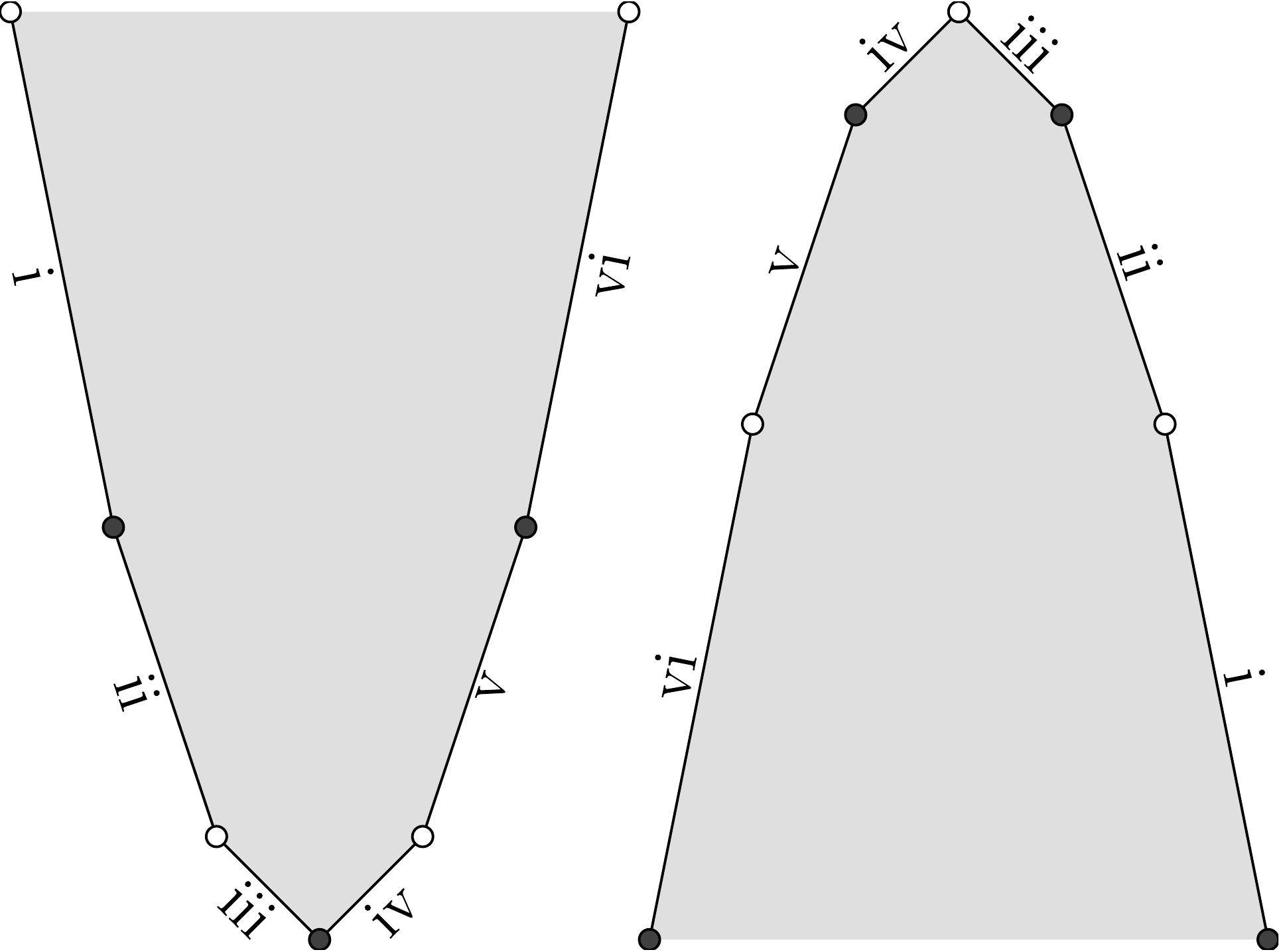}
\caption{The surface $S_1$ is built from two infinite polygons in the plane: The convex hulls of the sets
$\{(n,n^2)~:~n \in \Z\}$ and $\{(n,-n^2)~:~n \in \Z\}$. Roman numerals indicate edges glued by translations.}
\label{fig:s1}
\end{center}
\end{figure}

This study is motivated by work of Veech, \cite{V}, which shows that translation surfaces built in a similar manner from two regular polygons have special geometric and dynamical properties. We describe aspects of \cite{V} in section
\ref{sect:future}. In short, these surfaces exhibit affine symmetries analogous to the action of $\SL(2, \Z)$ on the square torus. 

The surface $S_1$ also has affine symmetries described by a lattice in $\SL(2, \R)$. Furthermore, we will explain how $S_1$ arises from a limit of Veech's surfaces built from regular polygons. However, previous geometric and dynamical theorems on such surfaces do not directly apply to $S_1$ because this surface has infinite area, infinite genus, and two cone singularities with infinite cone angle. 
This motivates the question: Does the infinite genus surface $S_1$ exhibit ``nice'' geometric and dynamical properties?
The purpose of this article is to explain some such nice properties of $S_1$.

We also aim to show that the surface $S_1$ belongs to a $1$-parameter family of surfaces, $\{S_c : c \geq 1\}$, 
with closely related geometric and dynamical properties. We mentioned above that $S_1$ is obtained as a limit of translation surfaces built from a pair of regular polygons. We obtain the family $S_c$ by analytically continuing this limiting process. See section \ref{sect:limit}.

The following points summarize the contents of this paper:
\begin{itemize}
\item We explain that the surfaces $S_c$ for $c \geq 1$ are pairwise homeomorphic, and that this homeomorphism is canonical up to isotopies which fix the singularities. (See Proposition \ref{prop:h}.)
\item We show that $S_1$ and $S_c$ have homotopic geodesics for all $c>1$. That is, given a geodesic $\gamma$ in $S_1$,
there is a geodesic $\gamma'$ in $S_c$ which is homotopic to the image of $\gamma$ under the canonical homeomorphism.
(See Theorems \ref{thm:same_geodesics} and \ref{thm:same_geodesics2} and Remark \ref{rem:Geodesics and codes}.)
\item We describe the {\em affine automorphism group} of $S_c$, $\Aff(S_c)$, for each $c \geq 1$. This is the group of homeomorphisms $S_c \to S_c$ which preserve the affine structure of $S_c$. (The group structure
of $\Aff(S_c)$ is provided by combining Theorem \ref{thm:veech_groups} with Proposition \ref{prop:bijection}.
Lemma \ref{lem:affine_automorphisms} describes the action of a collection of generators.)
\item We show that each affine automorphism of $S_c$ is isotopic to an affine automorphism of $S_1$, and vice versa. (See Theorem \ref{thm:isotopic_affine_action}.)
\item We show that the surface $S_1$ has the {\em lattice property}, i.e., the group of derivatives of orientation preserving affine automorphisms of $S_1$ form a lattice in $SL(2,\R)$. This group is the congruence two subgroup of $\SL(2, \Z)$.
(See Corollary \ref{cor:lattice property}.)
\end{itemize}

\begin{comment}
Let $\Sigma \subset S_1$ denote the set of two singularities. 
We call the surface $S_1$ a {\em translation surface}, because we can provide
an open cover $\{U_i\}$ of $S_1 \sm \Sigma$ and charts $\phi_i: U_i \to \R^2$ so that the transition functions
($\phi_j \circ \phi_i^{-1}$ restricted to $\phi_i(U_i \cap U_j)$) are translations. An {\em affine automorphism} of $S_1$
is a homeomorphism $\psi:S_1 \to S_1$ which acts locally as an element of the affine group of the plane. Using the local charts,
we obtain an identification between each tangent plane to $S_1 \sm \Sigma$ and $\R^2$. The action induced on $\R^2$ by
the action of $\psi$ on tangent planes is identical at all points. 
We call this action the {\em derivative of $\psi$} and denote it by $\D(\phi)$. Observe $\D(\phi) \in \GL(2,\R)$. 
The {\em affine automorphism group} $\Aff(S_1)$ is the group of affine automorphisms, and the {\em Veech group} is the group $\Gamma(S_1)=D \big(\Aff(S_1)\big)$. We give further background
on translation surfaces .
\end{comment}

This paper is structured as follows. In the following section, 
we provide context for this problem and suggest directions of future work.
In section \ref{sect:limit}, we construct the surfaces $S_c$, and explain how they relate to regular polygons. 
Then in section \ref{sect:background}, we provide necessary background on the subject of translation surfaces 
in the context of infinite surfaces. In section \ref{sect:results}, we give rigorous statements of the results mentioned above. We spend the remainder of the paper proving these statements.

\section{Context and future work}
\label{sect:future}
Motivation to study this particular translation surface comes from the classical theory of (closed) translation surfaces. We briefly outline some of this theory to explain this motivation. 
For a more detailed introduction see one of the surveys \cite{MT98}, \cite{Zorich06}, or \cite{Y10}.

Classically, a {\em translation surface} is a Riemann surface $X$ equipped with a holomorphic $1$-form, denoted $(X,\omega)$. The $1$-form provides local charts to $\C$ which are well defined up to translations. We can use these charts to define a metric on the surface by pulling back the metric from the plane. Near the zeros of $\omega$, this metric has cone singularities whose cone angles are multiples of $2\pi$. The direction of a vector is invariant under the geodesic flow on $(X,\omega)$. We define the {\em straight-line flow in direction $\theta$} on $(X,\omega)$ to be the flow defined in local coordinates by 
$F^t_\theta(z)=z+t e^{i \theta}.$
Trajectories of this flow are unit speed geodesics traveling in the direction $\theta$. 

We may fix the genus, and consider the moduli space $\Omega$ of translation surfaces. There is a well-known $\SL(2,\R)$ action on this space, which is closely related to the Teichm\"uller geodesic flow on the moduli space $\sM$ of Riemann surfaces. It is now well understood that an the structure $\SL(2,\R)$-orbit of $(X,\omega)$ inside $\Omega$ can be used to obtain asymptotic information about the geodesic flow on $(X,\omega)$. Veech's theorem provides a striking example of this phenomenon, involving the case when 
$(X,\omega)$ has the lattice property. As above, this means that $(X,\omega)$ is stabilized by a lattice $\Gamma \subset  \SL(2,\R)$. Equivalently, the $\SL(2,\R)$-orbit of $(X,\omega)$ descends to surface in $\sM$ which is Teichm\"uller isometric to the hyperbolic plane modulo $\Gamma$.
\begin{theorem}[Veech Dichotomy \cite{V}]
If $(X,\omega)$ has the lattice property, then for every $\theta$ the straight-line flow $F^t_\theta$
is either uniquely ergodic or {\em completely periodic} (every forward or backward trajectory which does not hit a singularity is periodic). 
\end{theorem}

In this paper, we see that our surface $S_1$ has the lattice property. But, it is not a closed surface,
so Veech's theorem does not apply. From results in this paper, it follows that when the direction $\theta$ has rational slope, either $F^t_\theta$ is completely periodic, or the the flow is dissipative and the surface
decomposes into infinite strips in direction $\theta$. For example, the horizontal direction is completely periodic, while the vertical direction decomposes into strips. 

We would like an ergodic theoretic understanding of the straight line flow in directions $\theta$ of irrational slope. 
Because the system is non-compact, it is too much to hope for unique ergodicity. Results in this paper imply
that a certain coding spaces of geodesics are identical for $S_1$ and $S_c$ for $c > 1$. 
See Theorem \ref{thm:same_geodesics2}.
Further work carried
out in the preprint \cite{Hinf} implies that in directions of irrational slope, there is an orbit 
equivalence between the flow $F^t_\theta$
on $S_1$ and a straight line flow in a different direction $\theta'$ on each surface $S_c$ for each $c \geq 1$.
This further work rests on some topological arguments and proof that the flows $F^t_\theta$ on $S_1$ 
and $F_{\theta'}^t$ on $S_c$ are recurrent, which is beyond the scope of this paper. 
Lebesgue measure on $S_c$ then gives rise to alternate $F^t_\theta$ invariant measures on $S_1$ via a pullback operation. 
We conjecture that all $F^t_\theta$-invariant ergodic measures arise in this way. Similar results are obtained for certain directions on a related class of surfaces in \cite{Hinf}, however we were unable to prove this result for the surface $S_1$. We hope to improve the argument to apply to $S_1$ in the future. 

Other examples of non-compact translation surfaces with the lattice property arise from covering constructions. See \cite{HS09} and \cite{HW10}, for instance. Using the work of \cite{ANSS02},
infinite branched cover of a torus is exhibited in \cite{HHW10}, where a trichotomy is exhibited for the Teichm\"uller flow. In every direction $\theta$ of rational slope, the straight-line flow is either completely periodic or the flow is dissipative and the surface decomposes into two strips. In directions of irrational slope, the invariant measures of $F^t_\theta$ are classified. These measures all arise from a similar construction involving another $1$-parameter family of translation surfaces with isotopic geodesics. I believe
the surface described in \cite{HHW10} is the only infinite translation surface for which such a result is known. The paper \cite{Hinf} provides an alternate proof. 

Our understanding of the ergodic theory of straight-line flows on infinite cyclic branched covers of translation surfaces is developing rapidly. See \cite{HWarxiv12}, \cite{FUpreprint}, \cite{RTarxiv11} and \cite{RTarxiv12} for instance. 

Finally, it is worth noting that many more basic facts become false or non-trivial in the setting of infinite translation surfaces. In the sequel to this paper \cite{Higl2}, the author will investigate the dynamical behaviour of hyperbolic affine automorphisms of the surface $S_1$. We will show that the action of a hyperbolic affine automorphism $\widehat{H}:S_1 \to S_1$ is non-recurrent in the sense that the conclusion of the Poincar\'e recurrence theorem fails to be true. Nonetheless, the automorphism satisfies a mixing-type result. Then there is a constant $\kappa$ depending on $\widehat{H}$ so that for every pair of cylinders, 
$$\lim_{m \to \infty} m^{\frac{3}{2}} \text{Area}\big(\widehat{H}^m(A) \cap B\big)=  \kappa \text{Area}(A) \text{Area}(B).$$
This result is also discussed in the preprint \cite{Higl}.

The surface $S_1$ also arises from a limiting process involving translation surfaces constructed from irrational polygonal billiard tables. The author hopes to develop this connection in a future paper.

\section{The limiting process}
\label{sect:limit}
Here is a dynamical way to describe a regular $n$-gon. Consider the rotation given by
$$R_t=\left[ \begin{array}{cc}
\cos t & - \sin t \\
\sin t & \cos t
\end{array} \right] \in \SO(2,\R).$$
The regular $n$-gon is the convex hull of the orbit of the point $(1,0)$ under the group generated by the rotation $R_{\frac{2 \pi}{n}}$. 

In order to take a limit we conjugate this rotation by the affine transform 
$C_t:(x,y)\mapsto(\frac{y}{\sin t}, \frac{x-1}{\cos t - 1} )$. The purpose of
$C_t$ is to normalize three vertices of the polygons. We have
$$C_t(1,0)=(0,0), ~~
C_t(\cos t, \sin t)=(1,1), ~~ \textrm{and} ~~
C_t(\cos t, -\sin t)=(-1,1).$$
Setting $c=\cos t$ and defining $T_c=C_t \circ R_t \circ C_t^{-1}$ yields
the affine map $T_c:\R^2 \to \R^2$ given by
\begin{equation}
\label{eq:generalized_rotation}
T_c:(x,y) \mapsto \big(c x+(c-1) y+1, (c+1)x+cy+1\big).
\end{equation}

\begin{figure}[b]
\begin{center}
\includegraphics[width=4in]{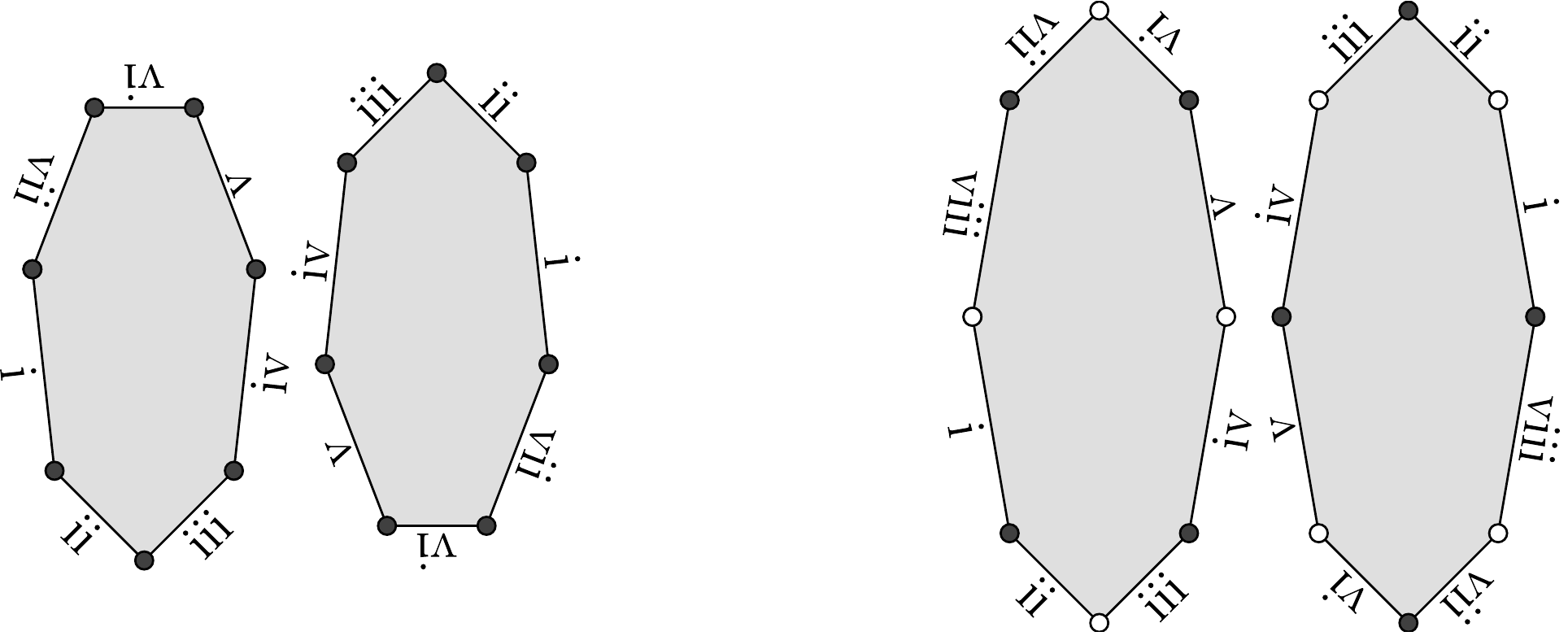}
\caption{The translation surface $S_{\cos \frac{2\pi}{7}}$ and $S_{\cos \frac{\pi}{4}}$ are built from pairs
of affinely regular polygons.}
\label{fig:veechsurfaces}
\end{center}
\end{figure}

Let $Q_c^+$ be the convex hull of the set of points $\{P_{c}^k=T_c^k (0,0)\}_{k \in \Z}$.
For $c=\cos \frac{2 \pi}{n}$, $Q_c^+$ is an affinely regular $n$-gon. 
For $c=1$ the collection of forward and backward orbits of $(0,0)$ is the set of points 
$\{(n,n^2) ~|~n \in \Z\}$, the integer points on the parabola
$y=x^2$. Finally for $c>1$,
the orbit of $(0,0)$ lies on a hyperbola. Assume $c=\cosh t$. 
Up to an affine transformation, the orbit of $(0,0)$ is $\{ (\cosh nt, \sinh nt) ~|~n \in \Z\}$.

We will use $Q_c^+$ to build our translation surfaces.
Let $Q_c^-$ be the image of $Q_c^+$ under a rotation
by $\pi$ around the origin. Each edge in $Q_c^+$ is parallel to its image in $Q_c^-$.
We identify each edge of $Q_c^+$ to its image edge in $Q_c^-$ by translation (rather than rotation). 
We call the resulting translation surface $S_c$. See figure \ref{fig:veechsurfaces} for some of the cases with $c < 1$. The case $S_1$
is drawn in figure \ref{fig:s1}, and $S_{\frac{5}{4}}$ is shown in figure \ref{fig:surface_cylinders}.

Observe that for each $k$, the map $c \mapsto P_{c}^k=T_c^k (0,0)$ is continuous. For this reason, we can think of 
the surface $S_1$ as a limit of the surfaces $S_{\cos \frac{2 \pi}{n}}$ as $n \to \infty$ and $\cos \frac{2 \pi}{n} \to 1$. 
Similarly, we view $c \mapsto S_c$ for $c \geq 1$ as a continuous deformation of translation surfaces. Concretely, we have the following:

\begin{proposition}[A family of homeomorphisms]
\label{prop:h}
There is a family of homeomorphisms $h_{c, c'}: S_{c} \to S_{c'}$
defined for $c \geq 1$ and $c' \geq 1$ which satisfy the following statements.
\begin{itemize}
\item $h_{c,c}$ is the identity map, and $h_{c,c'} \circ h_{c', c''}=h_{c,c''}$. 
\item $h_{c, c'}$ sends singular points to singular points.
\item $h_{c,c'}(Q_c^+)=Q_{c'}^+$ and $h_{c,c'}(Q_c^-)=Q_{c'}^-$.
\item Let $B$ be the bundle of the surfaces with singularities removed, $S_c \smallsetminus \Sigma$, over the ray $\{c~:~c \geq 1\}$. This
bundle is metrized to be locally isometric to $\R^3$. The map $B \times \{c'~:~c' \geq 1\} \to B$ which sends
the pair consisting of a point $x \in S_c$ and a $c' \geq 1$ to the point $h_{c,c'}(x)\in S_{c'}$ is continuous in the metric topology.
\end{itemize} 
\end{proposition}
\begin{proof}
To construct such a family of maps, we triangulate each $Q_c^\pm$ in the same combinatorial way, and then define
$h_{c,c'}$ piecewise, so that it affinely maps triangles to triangles.
\end{proof}

\begin{figure}[t]
\begin{center}
\includegraphics[width=4in]{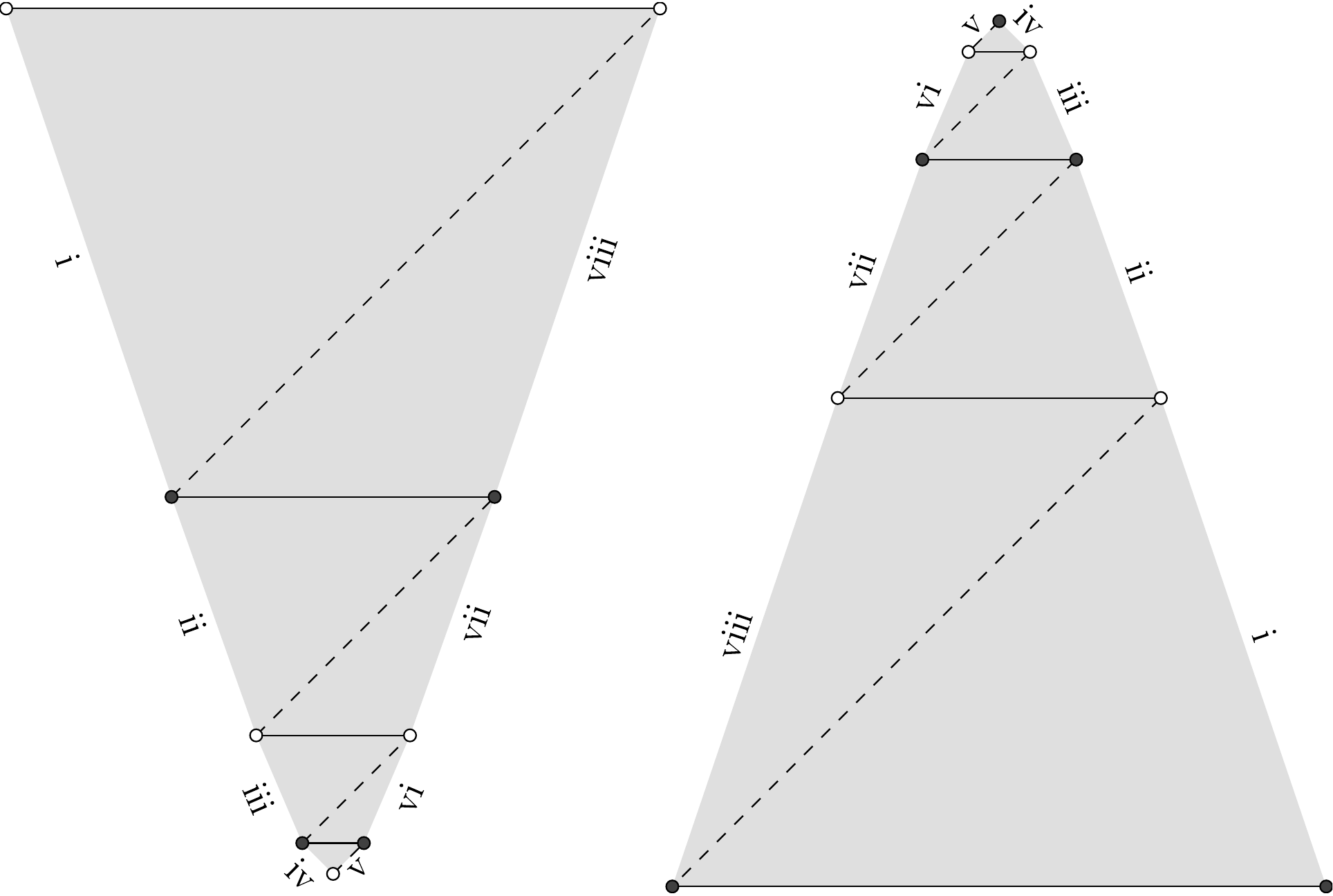}
\caption{The surface $S_c$ with $c=\frac{5}{4}$ is shown with some geodesic segments joining singularities.
%Horizontal and slope one saddle connections are drawn.
}
\label{fig:surface_cylinders}
\end{center}
\end{figure}

\section{Background on translation surfaces}
\label{sect:background}
Here we will briefly introduce some essential ideas in the subject of translation surfaces
in the context necessary to understand this paper. The treatment here differs from that of section \ref{sect:future}.
We strive for an elementary exposition, while emphasizing non-compact translation surfaces.

A {\em translation surface} $S$ is a collection of polygons in the plane with edges glued pairwise by translations. 
We insist that $S$ be connected, which implies that the collection of polygons is at most countable.
Any point in the interior of a polygon or in the interior of an edge has a neighbourhood with an injective coordinate chart to the plane, which is canonical up to post composition with a translation. 
Vertices of polygons are identified in $S$. We call the points of $S$ that arise in this way {\em singularities}. 
A singularity can be a cone singularities with cone angle which is an integer multiple of $2\pi$. Infinite cone angles can arise if infinitely many polygons are used (as for $S_1$).

Suppose $S$ is a translation surface and $\theta \in \R/2\pi \Z$ is a  direction. The {\em straight-line flow} on $S$ in the direction $\theta$ 
is the flow $F^t_\theta$ given in local coordinates by 
$$F^t_\theta(x,y)=(x,y)+t (\cos \theta, \sin \theta).$$
This flow is defined starting at any non-singular point of $S$. 
A forward trajectory has one of three possible behaviours. It might be that $F^t_\theta(x,y)$ is defined for all $t>0$.
It could be that there is a $t>0$ so that $F^t_\theta(x,y)$ is singular. In this case we don't define 
$F^{t'}_\theta(x,y)$ for $t'>t$. Finally, it could be that $F^t_\theta(x,y)$ crosses infinitely many edges
in finite time. In this case, the trajectory is defined only prior to the accumulation point of crossing times.
The same three possibilities hold for backward trajectories. 
The surfaces $S_c$ for $c \geq 1$ are all complete, so this third possibility never happens for this family. But, our arguments
work in this more general setting. 

Let $S$ and $S'$ be translation surfaces. A homeomorphism $\widehat A:S \to S'$ is called an {\em affine} if
in each local coordinates chart $\psi$ is of the form
$$\psi(x,y)=(a x+ by+t_1, cx+dy+t_2) \textrm{ with $A=\left[\begin{array}{rr} a & b \\ c & d\end{array}\right] \in \GL(2, \R)$ and $t_1,t_2 \in \R$}.$$
The constants $t_1$ and $t_2$ may depend on the chart. Because the transition functions are translations, the matrix $A$ is an invariant of $\psi$. 
We call this matrix the {\em derivative}, $\D(\psi)=A \in \GL(2, \R)$. 

There is a natural action of $\GL(2, \R)$ on translation surfaces. If $A \in \GL(2, \R)$ and $S$ is a translation surface, 
we define $A(S)$ by composing each coordinate chart with the corresponding linear map $A: \R^2 \to \R^2$. 

An {\em affine automorphism} of a translation surface $S$ is an affine homeomorphism $\widehat A:S \to S$. 
The collection of all affine automorphisms of $S$ form a group, called the affine automorphism group $\Aff(S)$. 
The group $\D\big(\Aff(S)\big) \subset \GL(2,\R)$ is called the {\em Veech group} of $S$ and is denoted $\Gamma(S)$. 
An alternate definition of the Veech group is given by 
$$\Gamma(S)=\{A \in \GL(2, \R) ~:~ \textrm{$\exists$ an affine homeomorphism $\psi: S \to A(S)$ with $\D(\psi)=I$}\}.$$

\section{Results}
\label{sect:results}

The following theorem describes the Veech groups of $S_c$.

\begin{theorem}[Veech groups]
\label{thm:veech_groups}
The Veech groups $\Gamma(S_c)\subset \GL(2,\R)$ for $c \geq 1$
are generated by the involutions $-I$,
$$A_c=\left[\begin{array}{cc}
-1 & 0 \\
0 & 1
\end{array}\right],
\quad
B_c=
\left[\begin{array}{cc}
-1 & 2 \\
0 & 1
\end{array}\right],
\quad \textrm{and} \quad
C_c=
\left[\begin{array}{cc}
-c & c-1 \\
-c-1 & c
\end{array}\right].
$$
\end{theorem}
For $c=\cos(\frac{2 \pi}{n})$, it is a theorem of Veech that the elements given above generate $\Gamma(S_c)$ \cite{V}, which is an $(\frac{n}{2}, \infty, \infty)$-triangle group when 
$n$ is even, and an $(n,2, \infty)$ triangle group when $n$ is odd.

Note in particular, the surface $S_1$ has the lattice property:
\begin{corollary}[The lattice property]
\label{cor:lattice property}
The orientation preserving part of $\Gamma(S_1)$ is the congruence two subgroup of $\SL(2, \Z)$. 
\end{corollary}

We describe the relations in this matrix group below.
For all $c$, the matrices $A_c$, $B_c$, and $C_c$ are involutions and act as reflections in geodesics in the hyperbolic plane, $\H^2$, 
when projectivized to elements of $\Isom(\H^2)\cong \PGL(2, \R)$. By the theorem, the groups $\Gamma(S_c)$ are all representations of the group
$$\G^\pm=(\Z_2 \ast \Z_2 \ast \Z_2) \oplus \Z_2=\langle A,B,C,-I ~|~ A^2=B^2=C^2=I \rangle.$$
The geodesics associated to $A_c$ and $C_c$ intersect at angle $\frac{2\pi}{n}$ when $c=\cos(\frac{2\pi}{n})$,
are asymptotic when $c=1$, and disjoint and non-asymptotic for $c>1$. See figure \ref{fig:veechgroup}.

\begin{figure}[ht]
\begin{center}
\includegraphics[width=4in]{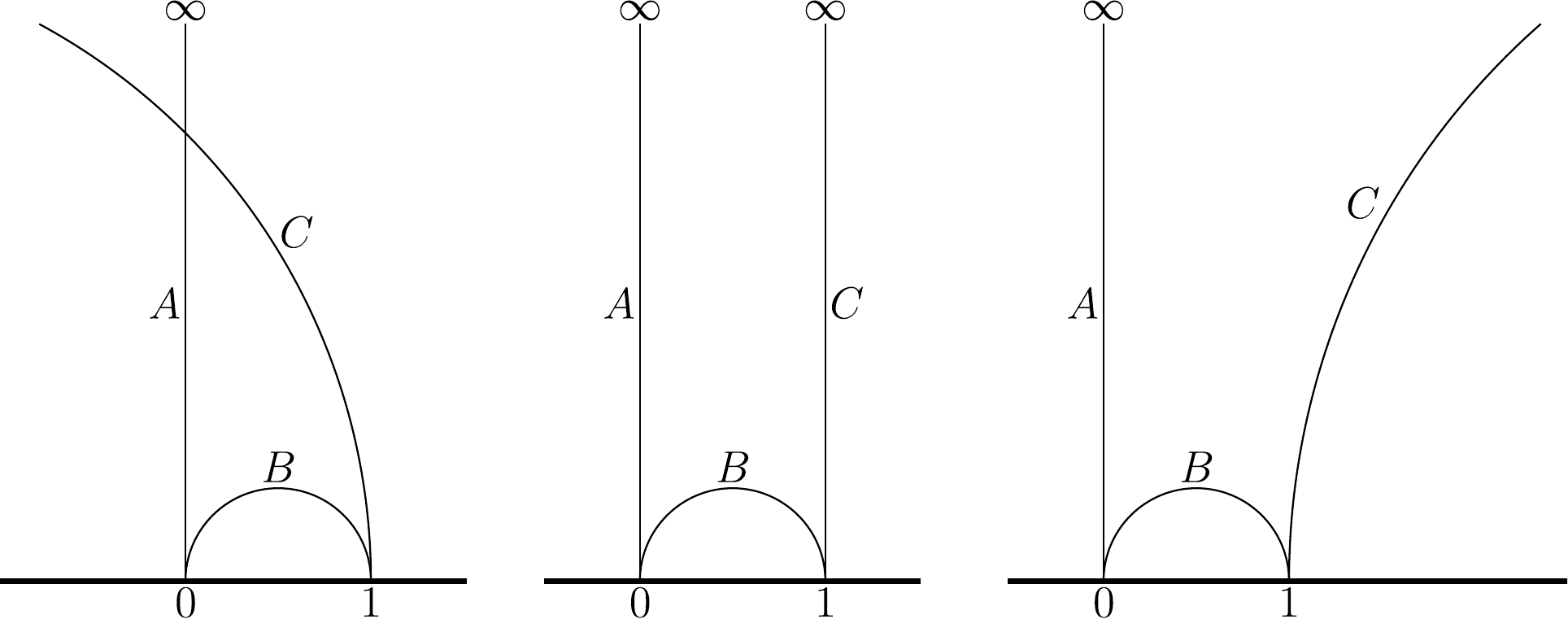}
\caption{This figure shows the geodesics in the upper half plane model of $\H^2$ 
that $A_c$, $B_c$, and $C_c$ reflect in for $c=\cos \frac{\pi}{4}$, $c=1$, and $c=\frac{5}{4}$ from left to right.}
\label{fig:veechgroup}
\end{center}
\end{figure}

When $c \geq 1$, the region in the hyperbolic plane bounded by the reflecting geodesics of $A_c$, $B_c$, and $C_c$ is a fundamental domain for the action of $\Gamma(S_c)$. Moreover, the representations $\G^\pm \to \Gamma(S_c) \subset \GL(2, \R)$ are faithful when $c \geq 1$. Both these facts follow from a variant of the standard Ping-pong Lemma in hyperbolic geometry (described in \cite{MT98}). 

We briefly describe the well known situation when $c < 1$ for completeness. When $n$ is even and $c=\cos(\frac{2 \pi}{n})$, the triangle formed by the reflecting geodesics is again a fundamental domain. Thus, $\Gamma(S_c)$ is isomorphic to $\G^\pm$ modulo the relation $(A_c C_c)^\frac{n}{2}=-I$. When $n$ is odd, the element $(A_c C_c)^{\lfloor \frac{n}{2} \rfloor} A_c$ reflects in a geodesic orthogonal to the reflecting geodesic of $B_c$. In this case $\Gamma(S_c)$ is isomorphic to $\G^\pm$ modulo the relations $(A_c C_c)^n=I$ and $[(A_c C_c)^{\lfloor \frac{n}{2} \rfloor} A_c, B_c]=-I$. 

\begin{proposition}
\label{prop:bijection}
For all $c \geq 1$, the map $\D:\Aff(S_c) \to \Gamma(S_c)$ is a bijection.
\end{proposition}

Because of this proposition, an affine automorphism is uniquely determined by its derivative. 
This allows us to introduce the following notation:

\begin{notation}
Recall $\Gamma(S_c) \cong \G^\pm$ when $c \geq 1$. Given $G \in \G^\pm$ and $c \geq 1$, we denote the corresponding element of $\Gamma(S_c) \subset \GL(2, \R)$ by $G_c$. Whenever the derivative map $\D:\Aff(S) \to \Gamma(S)$ is a bijection, given $A \in \Gamma(S)$,
we use $\widehat A \in \Aff(S)$ to denote the corresponding affine automorphism $\widehat A:S \to S$.
\end{notation}

We explicitly describe the topological action of generators for the affine automorphism group $\Aff(S_c)$ in Lemma \ref{lem:affine_automorphisms}. 
The following theorem uses the family of homeomorphisms $h_{c,c'}:S_c \to S_{c'}$ in 
Proposition \ref{prop:h} to say that the affine automorphism groups act on each $S_c$ in the same way.
Note that because the singular points of $S_c$ are infinite cone singularities, any homeomorphism $S_c \to S_{c'}$ must map singularities to singularities.
In particular, when two maps $S_c \to S_{c'}$ are isotopic, they are isotopic by an isotopy which preserves singularities.

\begin{theorem}[Isotopic Affine Actions]
\label{thm:isotopic_affine_action}
The homeomorphisms $S_c \to S_c'$ given by $h_{c,c'} \circ \widehat G_c$ and $\widehat G_{c'} \circ h_{c,c'}$ are isotopic
for all $G \in \G^\pm=\Gamma(S_c)$, $c \geq 1$ and $c' \geq 1$.
\end{theorem}

The remaining results explain that the surfaces have the same geodesics in a combinatorial sense.
A {\em saddle connection} in a translation surface $S$ is a geodesic segment joining singularities
with no singularities in its interior. We say two saddle connections $\sigma$ and $\tau$ are {\em disjoint} if $\sigma \cap \tau$ is contained in the set of endpoints.

\begin{theorem}[Isotopic triangulations]
\label{thm:isotopic_triangulations}
Suppose $\{\sigma_i\}_{i \in \Lambda}$ is a disjoint collection of saddle connections in $S_c$ for $c \geq 1$ which triangulate the surface.
Then for each $c' \geq 1$, there is a disjoint collection of saddle connections $\{\tau_i\}_{i \in \Lambda}$ and a homeomorphism
$S_c \to S_c'$ isotopic to $h_{c,c'}$ so that $\sigma_i \mapsto \tau_i$ for all $i \in \Lambda$.
\end{theorem}

We can use this theorem to show that all geodesics are the same combinatorially. 
Fix a translation surface $S$ and a collection of disjoint collection of saddle connections in $S$, $\sT_S=\{\sigma_i\}_{i \in \Lambda}$, which triangulate $S$. 
Say that an {\em interior geodesic} in a translation surface $S$ 
is a map $\gamma_0:I \to S$, where $I \subset \R$ is an open interval containing $0$,
which satisfies the following statements:
\begin{itemize}
\item In local coordinates, $\gamma_0(a+t)=\gamma_0(a)+t \u$
for some unit vector $\u \in \R^2$.
\item The image $\gamma_0(I)$ does not contain any singularities.
\item The path $\gamma_0$ passes through a countably infinite number of saddle connections in $\sT_S$ in both forward and backward time.
\end{itemize}
We say $\gamma_0$ {\em travels in the direction of $\u$.}
Suppose $\gamma_0$ is an interior geodesic in $S$ for which $\gamma_0(0) \in \bigcup_{i \in \Lambda} \sigma_i$. Then, the set 
$X=\gamma_0^{-1}\big(\bigcup_{i \in \Lambda} \sigma_i\big) \subset I$ is discrete, bi-infinite and contains zero. 
See the discussion of straight line trajectories in Section \ref{sect:background}.
Thus, there is a unique increasing bijection $\psi:\Z \to X$ so that $\psi(0)=0$. The {\em code of $\gamma_0$} is the bi-infinite sequence
$\langle e_n \in \Lambda \rangle_{n \in \Z}$ so that $\gamma_0 \circ \psi(n) \in \sigma_{e_n}$. Since $\gamma_0$ is interior, this sequence is unique.

Now consider the surface $S_c$. Let $\{\sigma_i\}_{i \in \Lambda}$ be a collection of saddle connections which triangulate $S_c$. Let $\Omega_c \subset \Lambda^\Z$ denote the closure (in the product topology) of the collection of all codes of interior geodesics on $S_c$. This collection $\Omega_c$ is a shift space on a countable alphabet, with the property that each symbol can only be proceeded
or succeeded by four symbols.

Theorem \ref{thm:isotopic_triangulations} indicates that for each $c' \geq 1$ there is a corresponding triangulation $\{\tau_i\}_{i \in \Lambda}$ of $S_{c'}$. We can use this triangulation to code interior geodesics
on $S_{c'}$. We again let $\Omega_{c'} \subset \Lambda^\Z$ denote the 
the closure of the collection of all codes of interior geodesics on $S_{c'}$.
In fact these sets are the same:

\begin{theorem}[Same codes for geodesics]
\label{thm:same_geodesics}
For each $c \geq 1$ and $c' \geq 1$, the two coding spaces $\Omega_c$ and $\Omega_{c'}$ are equal as subsets of $\Lambda^\Z$.
\end{theorem}

We will also consider codes of interior geodesics traveling in the direction of a fixed unit vector $\u$. For $c \geq 1$ and any $\u$, let $\Omega_{c,\u}$ denote the closure of the collection of codes of interior geodesics traveling in direction $\u$ on $S_c$. 

\begin{theorem}[Same directional codes]
\label{thm:same_geodesics2}
Let $\Circ$ denote the collection of unit vectors in $\R^2$. For each
$c>1$, there is a continuous monotonic degree one map $\varphi_c:\Circ \to \Circ$
such that for each unit vector $\u \in \Circ$, the choice of the unit vector $\u'=\varphi_c(\u)$ gives 
a coding space $\Omega_{1,\u'}$ which is equal to the coding space $\Omega_{c,\u}$ as a subset of $\Lambda^\Z$.
\end{theorem}

The map $\varphi_c$ is defined using the action of the Veech groups on $\Circ$.
See Proposition \ref{prop:varphi_c}.

Observe that Theorem \ref{thm:same_geodesics} is really a corollary of Theorem \ref{thm:same_geodesics2}
as $\Omega_c$ is the closure of $\bigcup_\u \Omega_{c,\u}$. This union is independent of $c$ since $\varphi_c$ is always injective.

Theorems \ref{thm:isotopic_triangulations}, \ref{thm:same_geodesics}, and \ref{thm:same_geodesics2} hold more generally for deformations of translation surfaces with certain properties. See Lemmas \ref{lem:same_saddles}
and \ref{lem:same_geodesics2}.

\begin{remark}[Geodesics and codes]
\label{rem:Geodesics and codes}
We remark that an understanding of all the coding spaces $\Omega_{c,\u}$
for unit vectors $\u$ give a strong information about the coding of all geodesics,
not only interior geodesics. For instance, the codes of all saddle connections 
on $S_c$ in direction $\u$ which are not in the triangulation can be determined by $\Omega_{c,u}$. Namely, a word $w$ in the alphabet $\Lambda$ is the code of a saddle connection, if and only if $w$ appears in $\Omega_{c,u}$, and whenever $w$ appears, it is as a subword of
$\lambda_1 w \lambda_2$ or $\lambda_3 w \lambda_4$ with $\lambda_1,\ldots, \lambda_4 \in \Lambda$ and 
where $\lambda_1 \neq \lambda_3$
and $\lambda_2 \neq \lambda_4$. A {\em separatrix} is a straight-line trajectory leaving a singularity. The codes of sepeatrices in direction $\u$ which are not saddle connections can also be determined by looking at $\Omega_{c,u}$.
Geodesics with the same codes in $S_c$ and $S_c'$ are homotopic in the sense mentioned in the introduction.
\end{remark}

%NEXT SECTION
\section{The affine automorphisms}
\label{sect:automorphisms}

In this section, we find and describe elements of the affine automorphism groups of the surfaces $S_c$ defined in the previous section. At this point, we cannot assume Theorem \ref{thm:veech_groups}, which described the generators of the Veech group.
So, we use $\G_c^\pm \subset \GL(2, \R)$ to denote the group generated by $-I$, $A_c$, $B_c$, and $C_c$. Our description of affine automorphisms implies that that $\G_c^\pm \subset \Gamma(S_c)$. It also proves
the Isotopic Affine Action Theorem, assuming $G_c \in \Gamma(S_c)$ implies $G_c \in \G^\pm_c$ (which turns out to be true).

To ensure our notation for affine automorphisms used in the previous section makes sense we must prove Proposition \ref{prop:bijection}.
\begin{proof}[Proof of Proposition \ref{prop:bijection}]
Suppose $\psi \in \Aff(S_c)$ satisfies $\D(\psi)=I$. Then, $\psi$ maps saddle connections to saddle connections, and preserves their slope and length. 
In each surface $S_c$, there is only one saddle collection of slope one with length $\sqrt{2}$. Therefore, $\psi$ must fix all points on this saddle connection. Since
$\psi$ fixes a non-singular point and $\D(\psi)=I$, $\psi$ must be the identity map.
\end{proof}

It is useful to
work with alternate generators for $\G^\pm$. Define the elements $D=B A$, $E=(-I) C B$. The elements $\{A,D,E,-I\}$ also are
generators for $\G^\pm$. The corresponding matrices in $\G_c^\pm$ are given by
\begin{equation}
\label{eq:affop}
D_c=\left[\begin{array}{cc}
1 & 2 \\
0 & 1
\end{array}\right]
\quad
E_c=\left[\begin{array}{cc}
-c & c+1 \\
-c-1 & c+2
\end{array}\right]
\end{equation}
Note that $D_c$ and $E_c$ are orientation preserving parabolics.

\begin{lemma}[Affine Automorphisms]
\label{lem:affine_automorphisms}
For $c \geq 1$, $\G_c^\pm \subset \Gamma(S_c)$. Moreover, the affine automorphisms corresponding to generators of $-I, A_c, D_c, E_c \in \G_c^\pm$ may be described topologically (up to isotopy) as follows.
\begin{itemize}
\item $\widehat{-I}_c$ swaps the two pieces $Q_c^+$ and $Q_c^-$ of $S_c$, rotating each piece by $\pi$.
\item $\hat{A}_c$ is the automorphism induced by the Euclidean reflection in the vertical line $x=0$, which
preserves the pieces $Q_c^+$ and $Q_c^-$ of $S_c$ and preserves the gluing relations.
\item $\hat{D}_c$ preserves the decomposition of $S_c$ into maximal horizontal cylinders, and acts as a single right Dehn twist in each cylinder.
\item $\hat{E}_c$ preserves the decomposition of $S_c$ into maximal cylinders of slope $1$, and acts as a single right Dehn twist in each cylinder.
\end{itemize}
\end{lemma}

\begin{remark}
The automorphism $\hat{F}_c$ corresponding to the element
$$F_c=C_c A_c=\left[ \begin{array}{cc}
c & c -1 \\
c+1 & c
\end{array}\right]$$ 
may be of special interest. The action of $\hat{F}_c$ preserves the decomposition into two pieces,
$Q_c^+$ and $Q_c^-$.
It acts on the top piece as $T_c$ acts on the plane. (See equation \ref{eq:generalized_rotation}). 
When $c \geq 1$, the surface $S_c$ decomposes into a countable number of maximal 
strips in each eigendirection of $F_c$. The action of $\hat{F}_c$ preserves this decomposition into strips.
We number each strip by integers, so that each strip numbered by $n$ is adjacent to the 
strips with numbers $n \pm 1$. This numbering can be chosen so that 
the action of $\hat{F}_c$ sends each strip numbered by $n$ to the strip numbered 
$n+1$. So, the action of $\hat{F}_c$ for $c \geq 1$ is as nonrecurrent as possible. Given any compact set 
$K \subset S_c \smallsetminus \Sigma$, there is an $N$ so that for $n>N$, $\hat{F}_c^n(K) \cap K = \nullset$.
\end{remark}

We begin by stating a well known result that gives a way to detect parabolic elements inside the
Veech group. The idea is that a Dehn twist may be performed in a cylinder by a parabolic. See
figure \ref{fig:dehn}.
A cylinder is a subset of a translation surface isometric to $\R/k\Z \times [0,h]$. The ratio $\frac{h}{k}$ is called the {\em modulus} of the cylinder.

\begin{proposition}[Veech {\cite[\S 9]{V}}]
\label{prop:parabolic}
Suppose a translation surface has a decomposition into cylinders $\{C_i\}_{i \in \Lambda}$ 
in a direction $\theta$. Suppose further there is a real number $m\neq 0$ such that for 
every cylinder $C_i$, the modulus of $M_i$ of $C_i$ satisfies $m M_i \in \Z$.
Then, there is an affine automorphism of the
translation surface preserving the direction $\theta$, fixing each point on the boundary of each cylinder, 
and acting as an $mM_i$ power of single right Dehn twist in each cylinder $C_i$. The derivative this affine 
automorphism is the parabolic 
$$R_\theta \circ \left[\begin{array}{rr} 1 & m \\ 0 & 1\end{array}\right] \circ R_\theta^{-1},$$
where $R_\theta \in \SO(2)$ rotates the horizontal direction to direction $\theta$.
\end{proposition}

\begin{figure}[ht]
\begin{center}
\includegraphics[width=4.5in]{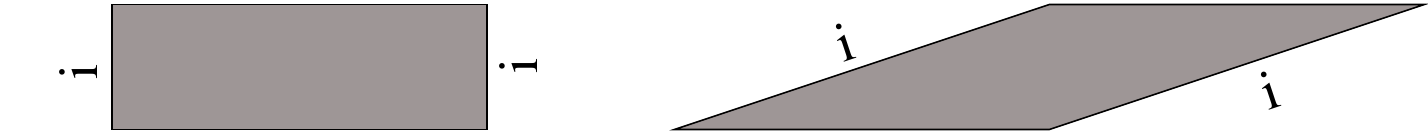}
\caption{The right cylinder is obtained by applying a shear to the left cylinder. There is an affine homeomorphism
from the right cylinder to the left with derivative $I$. The composition of these maps is used in Proposition \ref{prop:parabolic}.}
\label{fig:dehn}
\end{center}
\end{figure}

In the cases of $\widehat D_c$ and $\widehat E_c$, each $M_i$ will be equal, hence we get an affine automorphism which acts by a single right 
Dehn twist in each cylinder.

\begin{proof}[Proof of Lemma \ref{lem:affine_automorphisms}]
Recall, the surface $S_c$ for $c \geq 1$ was built from two pieces $Q_c^+$ and $Q_c^-$. We defined 
$Q_c^+$ to be the convex hull of the vertices 
$P_i=T_c^i(0,0)$ for $i \in \Z$, with $T_c$ as in equation \ref{eq:generalized_rotation}.
Next $Q_c^-$ was defined to be $Q_c^+$ rotated by $\pi$. $S_c$ is built by gluing the edges of $Q_c^+$
to its image under $Q_c^-$ by parallel translation. Indeed, it is obvious from this definition 
that the rotation by $\pi$ which swaps $Q_c^+$ and $Q_c^-$ restricts to an affine automorphism of the surface, $\widehat{-I}_c \in \Aff(S_c)$. The derivative of $\widehat{-I}_c$ is $-I_c=-I$, which therefore lies in $\Gamma(S_c)$.

Now we will see that the reflection in the line $x=0$ induces an affine automorphism ($\hat{A}$). The reflection
is given the map $r:(x,y)\mapsto(-x,y)$. $Q_c^+$ is preserved because $r(P_i)=P_{-i}$, which follows from the fact
that $r \circ T_c \circ r^{-1} = T_c^{-1}$. The reflection acts in the same way on $Q_c^-$, and thus preserves
gluing relations. Thus, $\hat{A}_c$ is an affine automorphism and its derivative, $A_c$, lies in the Veech group.

We will show that each cylinder in the horizontal cylinder decomposition has the same modulus,
which will prove that $\hat{D}_c$ lies in the affine automorphism group by proposition \ref{prop:parabolic}. Let
$P_i=(x_i,y_i)$. The circumference of the $n-th$ cylinder numbered vertically is given by
$C_n=2 x_{n-1}+2 x_{n}$, and the height is $H_n=y_n-y_{n-1}$. Now let $(x_{n-1},y_{n-1})=(\hat x, \hat y)$, so that
by definition of $T_c$, we have $(x_{n},y_{n})=(c \hat{x}+(c-1)\hat{y}+1,(c+1)\hat{x}+c \hat{y}+1)$. This makes
$$C_n=2(c+1)\hat{x}+2(c-1)\hat{y}+2 \quad \textrm{and} \quad H_n=(c+1)\hat{x}+(c-1)\hat{y}+1.$$
So that the modulus of each cylinder is $\frac{1}{2}$. It can be checked that the parabolic fixing the horizontal direction
and acting as a single right Dehn twist in cylinders of modulus $\frac{1}{2}$ is given by $D_c$.

It is not immediately obvious that there is a decomposition into cylinders in the slope $1$ direction. To see this,
note that there is only one eigendirection corresponding to eigenvalue $-1$ 
of the $\SL(2,\R)$ part of the affine transformation
$$U:(x,y) \mapsto (-cx+(c-1)y+1,-(c+1)x+cy+1)$$
is the slope one direction. It also has the property that $U \circ T_c \circ U^{-1}=T_c^{-1}$, which can be used to show that $U$ swaps $P_i$ with $P_{1-i}$. Therefore segment $\overline{P_{1-i} P_i}$ always has slope one.
The $n$-th slope one 
cylinder is formed by considering the union of trapezoid
obtained by taking the convex hull of the points $P_{n}$, $P_{n+1}$, $P_{1-n}$ and $P_{-n}$ and the same trapezoid rotated by $\pi$ inside $Q_c^-$. Now we will show that the moduli of these cylinders are all equal.
The circumference and height of the $n$-th cylinder in this direction is given below.
$$C_n=\sqrt{2}(x_{n}-x_{1-n}+x_{n+1}-x_{-n})$$
$$H_n=\frac{\sqrt{2}}{2}(x_{n+1}-x_{n},y_{n+1}-y_{n}) \cdot (-1,1)$$
Let $P_n=(\hat{x},\hat{y})$. Then $P_{-n}=(-\hat{x},\hat{y})$. We also have:
$$P_{n+1}=\big(c \hat{x}+(c-1)\hat{y}+1,(c+1)\hat{x}+c \hat{y}+1\big).$$
$$P_{1-n}=\big((-c)\hat{x}+(c-1)\hat{y}+1, (-c-1)\hat{x}+c \hat{y}+1\big).$$
By a calculation, we see that $C_n=\sqrt{2}(2c+2)\hat{x}$ and $H_n=\sqrt{2}\hat{x}$. So, the modulus of each cylinder is $\frac{1}{2c+2}$. Thus by proposition \ref{prop:parabolic},
$\hat{E}_c$ lies in the affine automorphism group. We leave it to the reader to check that the derivative 
of $\hat{E}_c$ must be $E_c$. 
\end{proof}

We now prove the Isotopic Affine Actions Theorem, assuming Theorem \ref{thm:veech_groups} which classifies the Veech group.
\begin{proof}[Proof of Theorem \ref{thm:isotopic_affine_action}.]
It is enough to prove the statement for the generators $-I, A, D, E \in \G^\pm$. Let $G$ be one of these generators.
It can be observed that the affine actions $\widehat G_c:S_c \to S_c$ act continuously on the bundle $B$ of surfaces $S_c$ over $\{c~:~c \geq 1\}$. We must show that $h_{c,c'} \circ \widehat G_c$ and $\widehat G_{c'} \circ h_{c,c'}$ are isotopic.
Let $c''\geq 1$ be a number between $c$ and $c'$. Consider the map
$\phi_{c''}:S_c \to S_c'$ given by 
$\phi_{c''}=h_{c'',c'} \circ \widehat G_{c''} \circ h_{c,c''}$. Continuously moving $c''$ from $c$ to $c'$
yields the desired isotopy.
\end{proof}

\section{A classification of saddle connections}
\label{sect:saddles}
In this section, we will classify the directions in $S_c$ where saddle connections can appear. We begin with $S_1$.

We use the notation $\frac{p}{q} \equiv \frac{r}{s} \pmod{2}$ to say that once the fractions are reduced 
to $\frac{p'}{q'}$ and $\frac{r'}{s'}$ so that numerator
and denominator are relatively prime, we have $p' \equiv r' \pmod{2}$ and $q' \equiv s' \pmod{2}$. 
We use $\frac{p}{q} \not\equiv \frac{r}{s} \pmod{2}$ to denote the negation of this statement.

In the statement of the following proposition, we use the concept of the holonomy of a saddle connection. Given any  
path $\gamma:[0,1] \to S$ in a translation surface which avoids the singularities on $(0,1)$,
there is a {\em development} of $\gamma$ into the plane. This is a curve $\dev(\gamma):[0,1] \to \R^2$ up to post-composition by with a translation, defined by following the local charts from $S$ to the plane. The {\em holonomy vector}
$\hol(\gamma)$ is obtained by subtracting the endpoint of $\dev(\gamma)$ from its starting point. The quantity
$\hol(\gamma)$ is invariant under homotopies which fix the endpoints. The notions of holonomy and the developing map are common in the world of
$(G,X)$ structures; see section $3.4$ of \cite{Thurston}, for instance.

\begin{proposition}[Saddle connections of $S_1$]
\label{prop:S_1_saddles}
Saddle connections $\sigma \subset S_1$ must have integral holonomy $\hol_1(\sigma)\in \Z^2$. 
A direction contains saddle connections if and only if it has rational slope, $\frac{p}{q}$, 
with $\frac{p}{q} \not \equiv \frac{1}{0} \pmod{2}$.
\end{proposition}
\begin{proof}
The holonomy of a saddle connection must be integral, because the surface $S_1$ was built from two (infinite) polygons
with integer vertices. The subgroup $\G_1^\pm \subset \Gamma(S_1)$ (generated by $A_1$, $D_1$, $E_1$, and $-I_1$)
is the congruence two subgroup of
$\widehat \SL^\pm(2,\Z)$. Thus, the linear action of  $\G_1^\pm$ on the plane preserves the collection of vectors
$$\textit{RP}=\{(p,q)\in \Z^2\smallsetminus\{(0,0)\}~|~\textrm{$p$ and $q$ are relatively prime}\}.$$
Furthermore, the orbits of $(0,1)$, $(1,1)$, and $(1,0)$ under $\Gamma(S_1)$ are disjoint and cover $\textit{RP}$. 
Thus, up to the affine automorphism group, the geodesic flow in a direction of rational slope looks like the geodesic flow in the horizontal, slope one, 
or vertical directions. There are saddle connections in both the horizontal and slope one directions, but not in the vertical direction.
Therefore, rational directions contain saddle connections unless they are in the orbit of the vertical direction under 
$\G_1^\pm$.
\end{proof}

In order to make a similar statement for $S_c$, we will need to describe the directions that contain saddle connections. 
We will find it useful to note that there is a natural bijective correspondence between directions in the plane modulo rotation by $\pi$, and the boundary of the hyperbolic plane $\partial \H^2$. 
This can be seen group theoretically. Directions in the plane
correspond to $\Circ= \SL(2,\R)/\textrm{H}$ where 
$$\textrm{H}=\{G \in \SL(2,\R)~|~G\big(\left[\begin{array}{c}1 \\ 0\end{array}\right]\big)= 
\left[\begin{array}{c}\lambda \\ 0\end{array}\right] \textrm{ for some $\lambda > 0$}\}.$$
Both directions mod rotation by $\pi$ and the boundary of the hyperbolic plane correspond to
the real projective line, $\RP^1=\SL(2,\R)/\textrm{H}^\pm$, where 
$$\textrm{H}^\pm=\{G \in \SL(2,\R)~|~G\big(\left[\begin{array}{c}1 \\ 0\end{array}\right]\big)= 
\left[\begin{array}{c}\lambda \\ 0\end{array}\right] \textrm{ for some $\lambda\neq 0$}\}.$$

Let $\Circ=(\R^2\smallsetminus \{(0,0)\})/\R_{>0}$, be the collection of rays leaving the origin.
Consider the left action of the groups $\G^\pm_c$ on $\Circ$. We have the following.
\begin{proposition}[Semi-conjugate actions]
\label{prop:varphi_c}
For all $c > 1$, there is a continuous (non-strictly) monotonic map $\varphi_c:\Circ \to \Circ$ of degree one so that the following
diagram commutes for all $G \in \G^\pm$.
$$
\begin{CD}
\Circ @>G_c>> \Circ \\
@VV\varphi_cV @VV\varphi_cV 
\\
\Circ @>G_1>> \Circ
\end{CD}$$
We may also assume that $\varphi_c$ preserves the horizontal ray $\{(x,0)~:x>0\}$ and the slope one ray $\{(x,x)~:~x >0\}$. The map $\varphi_c$ commutes with the rotation of the plane by $\pi$.
\end{proposition}

\begin{proof}
Existence of this map follows from \cite{Ghys87}, for instance. The following is a more natural proof using hyperbolic geometry.

Let $\G^+ \subset \G^\pm$ be those elements $G \in \G^\pm$ for which each $\det~G_c=1$. This is an index two subgroup and isomorphic to the product of the free group with two generators with $\Z/2\Z$.
We let $\G^+_c=\{G_c~:~G \in \G^+\} \subset \SL(2, \R)$. And use $P \G^+_c \subset \PSL(2, \R)$ to denote the projectivized groups.
The surfaces $\Sigma_c=\H^2/\G_c^+$ and $\Sigma_1=\H^2/\G_1^+$ are thrice punctured spheres whose fundamental groups are canonically identified with $\G^+$. 
Let $\psi: \Sigma_c \to \Sigma_1$ be a homeomorphism which induces the trivial map between the fundamental groups (as identified with $\G^+$). 
We may choose $\psi$ so that it is invariant under the action of $\G^\pm/\G^+$ (which acts on each surface as a reflective symmetry).
Since the fundamental groups are identified, there is a canonical lift to a map between the universal covers $\widetilde \psi:\widetilde \Sigma_c \to \widetilde \Sigma_1$
so that $\psi \circ G=G \circ \psi$ for all $G \in \G^+$, where $G$ is acting on the universal covers as an element of the covering group. 
This also holds for elements $G \in \G^\pm$ because of the $\G^\pm/\G^+$ invariance of $\psi$. Now noting the
canonical identification of these universal covers with $\H^2$ we have $\widetilde \psi:\H^2 \to \H^2$ so that $\widetilde \psi \circ G_c=G_1 \circ \widetilde \psi$ for all
$G \in \G^+$. This map induces a continuous monotonic degree $1$ map on the boundary of the hyperbolic plane $\RP^1$. The desired map $\varphi_c$ is a lift of this map to the double cover $\Circ$ of $\RP^1$.
\end{proof}

The map $\varphi_c$ is reminiscent of the famous devil's staircase, a continuous surjective map 
$[0,1] \to [0,1]$ which contracts intervals in the compliment of a Cantor set to points.
Indeed, the limit set $\Lambda_c$ of the group $\G^\pm_c$ is a $\G^\pm_c$-invariant Cantor set, and
the connected components of the domain of discontinuity, $\RP^1 \smallsetminus \Lambda_c$, 
are contracted to points by $\varphi_c$.

We will see that the saddle connections in $S_c$ and in $S_1$ are topologically the same. We will now make this notion rigorous. 
Given a path $\gamma:[0,1] \to S_c$, we use $[\gamma]$ to denote the equivalence class of paths which are homotopic to $\gamma$ relative to their endpoints. We do not allow these homotopies to pass through singular points. 

\begin{theorem}[Classification of saddle connections]
\label{thm:classification_of_saddle_connections}
There is a saddle connection in direction $\theta \in \Circ$ on $S_c$ for $c > 1$ if and only if 
there is a saddle connection in the direction 
$\varphi_c(\theta)$ on $S_1$. 
Equivalently, $\theta$ contains saddle connections if and only if $\theta$ is an image of
the horizontal or slope one direction under an element of $\G^\pm_c=\langle -I_c, A_c, D_c, E_c \rangle$. 
Furthermore, the collection of homotopy classes containing saddle connections are identical in $S_c$ and $S_1$. That is, for all saddle connections 
$\sigma \subset S_c$ there is a saddle connection in the homotopy class $[h_{c,1}(\sigma)]$ in $S_1$, and for
all saddle connections $\sigma' \subset S_1$ there is a saddle connection in the homotopy class $[h_{1,c}(\sigma')]$ in $S_c$.
\end{theorem}

We will prove this theorem by first proving a more abstract lemma. Then we will demonstrate that $S_c$ and $S_1$ satisfy the 
conditions of the lemma. We need the following two definitions.

The {\em wedge product} between two vectors in $\R^2$ is given by
\begin{equation}
\label{eq:wedge}
(a,b) \wedge (c,d)=ad-bc.
\end{equation}
This is the signed area of the parallelogram formed by the two vectors.

The function $\textit{sign}:\R \to \{-1,0,1\}$ assigns one to positive numbers, zero to zero, and $-1$ to negative numbers.

\begin{lemma}
\label{lem:same_saddles}
Let $h:S \to T$ be a homeomorphism between translation surfaces satisfying the following statements.
\begin{enumerate}
\item $S$ admits a triangulation by saddle connections.
\item For every saddle connection $\sigma \subset S$ the homotopy class $[h(\sigma)]$ contains a saddle connection
of $T$.
\item Every pair of saddle connections $\sigma_1, \sigma_2 \subset S$ satisfies
$$\textit{sign} \big( \hol(\sigma_1) \wedge \hol(\sigma_2) \big) = 
\textit{sign} \big( \hol ( h (\sigma_1)) \wedge \hol(h(\sigma_2)) \big).$$
\end{enumerate}
Then, for every saddle connection $\sigma \subset T$, the homotopy class $[h^{-1}(\sigma)]$ contains a saddle connection
of $S$.
\end{lemma}

In order to aid the proof of this lemma and a later lemma, we will prove a technical
proposition. To state this proposition, consider a translation surface $S$
triangulated by a collection of saddle connections $\sT=\{\sigma_i\}_{i \in \Lambda}$. 
Choose an integer $K \geq 1$.
Let $\Delta_0,\ldots, \Delta_K$ be triangles from the triangulation $\sT$,
and $\sigma_1,\ldots, \sigma_{K} \in \sT$ be saddle connections so that:
\begin{enumerate}
\item For each $j=1,\ldots, K$, we have that $\sigma_j$ is a common boundary of $\Delta_{j-1}$ and $\Delta_j$.
\item The sequence is non-backtracking, i.e. we never have
$\Delta_j=\Delta_{j+1}$ and we never have $\sigma_j = \sigma_{j+1}$. 
\end{enumerate}
Let $\til \Delta_0,\ldots, \til \Delta_K$ and $\til \sigma_1,\ldots, \til \sigma_j$ be lifts
of these objects to the universal cover branched over the singularities, $\til S$, 
which still satisfy statements (1) and (2) above. (We form such a cover by removing the singularties,
then taking the universal cover, and then placing the singularities back in the space.)
Choose a developing map, $\dev:\til S \to \R^2$, for the translation structure on $S$.
In analogy with polygonal billiards, we define the {\em unfolding} of $\gamma$ 
to be the image $U=\dev(\til \Delta_0 \cup \ldots \til \Delta_K)$. 
For each $j=1,\ldots,K$, let $s_j=\dev(\til \sigma_j)$ and let $I_j \subset \R$ denote the closed interval
\begin{equation}
\label{eq:I}
I_j=\{y~:~\text{there is an $x$ so that $(x,y) \in s_j$}\}.
\end{equation}
This situation is depicted in figure \ref{fig:saddleconnection}.

\begin{proposition}[Lifting developments]
\label{prop:dev}
Assume the notation of the previous paragraph. Let $I=\cap_{j=1}^K I_j$,
and assume that this is a non-degenerate closed interval. Define the subset of the unfolding
$V=\{(x,y) \in U: y \in I \}.$ Then,
\begin{itemize}
\item The set $V$ is convex.
\item Let $W=\dev^{-1}(V) \cap \bigcup_{j=0}^K \til \Delta_k$. The restriction
of the developing map, $\dev$, to $W$ is a homeomorphism onto $V$. 
\end{itemize}
\end{proposition}

We note that the developing map restricted to $\bigcup_{j=0}^K \til \Delta_k$,
is frequently not a homeomorphism as illustrated by Figure \ref{fig:saddleconnection}.

\begin{figure}[ht]
\begin{center}
\includegraphics[width=4.5in]{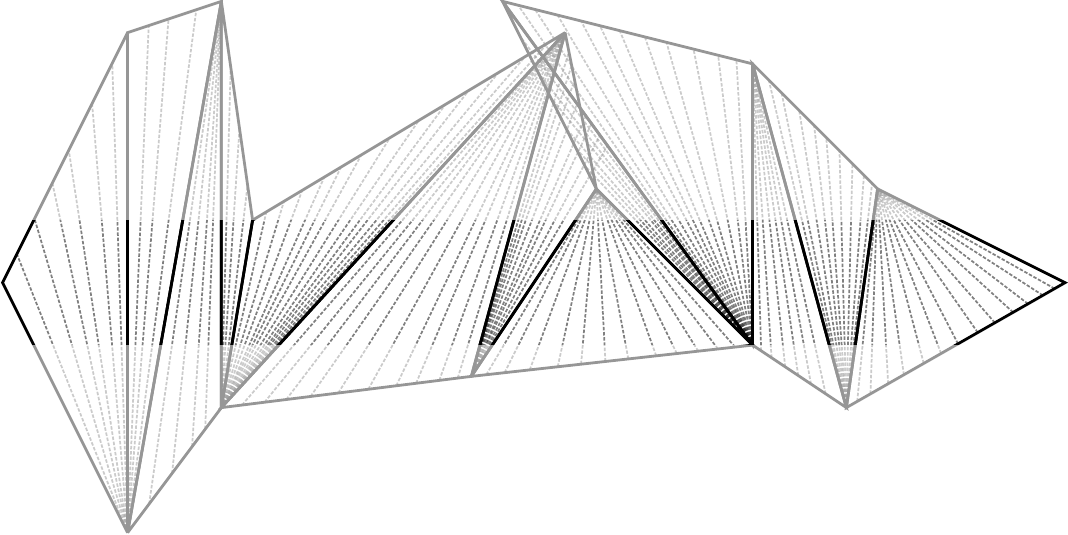}
\caption{The union of triangles $U$. The region $V \subset U$ is drawn darker than the rest of $U$. The developed image of the foliation $\sF$ relevant to the proof of Proposition \ref{prop:dev} is also shown.}
\label{fig:saddleconnection}
\end{center}
\end{figure}

\begin{proof}
Notice that the no backtracking assumption guarantees that the sequence
of triangles $\til \Delta_0,\ldots, \til \Delta_K$ on the universal cover branched cover, $\til S$, is non-repeating.
Extend the sequence of saddle connections
$\til \sigma_1, \ldots, \til \sigma_K$ to a family $\sF$ of non-horizontal line segments through vertices
of the triangles covering $\bigcup_{j=0}^K \til \Delta_j$ which is locally a foliation away from the vertices of the triangles. (The family is uniquely determined by the fact it passes through vertices
except for the first and last triangle, where there are two possible choices of
vertices for the segments to pass through. We can always make such choices so that each developed line segment is non-horizontal.)
Because the sequence of triangles is non-backtracking, the developed foliation always makes rightward progress as we sweep out the foliation in the direction of either increasing or decreasing triangle index. 
See figure \ref{fig:saddleconnection}. 
It follows that the image of the family $\sF$ under $\dev$ foliates the interior of $V$. 
Now observe that a point $p \in V$ is uniquely determined by a segment $\ell$ of $\sF$ so that $p \in \dev(\ell)$ and the $y$-coordinate of $p$. 
It follows that $\dev$ restricted to the preimage of $V$ inside of $\bigcup_{j=0}^K \til \Delta_k$ is one-to-one. This suffices to prove that the indicated restriction is a homeomorphism. 

The region $V$ can be described as the intersection of a horizontal strip
and four half planes bounded by lines extending the edges of $\Delta_0$ and $\Delta_K$. Thus it is convex. See figure \ref{fig:saddleconnection}.
\end{proof}

\begin{proof}[Proof of Lemma \ref{lem:same_saddles}]
Let ${\mathcal T}_S$ be the triangulation of $S$ by saddle connections given to us by statement (1) of the lemma. 
By statement (2), we we can straighten $h({\mathcal T}_S)$ to a triangulation ${\mathcal T}_T$ of $T$ by saddle connections. 

We define the {\em complexity} of a saddle connection $\tau \subset T$ relative to the triangulation ${\mathcal T}_T$ to be the number of
times $\tau$ crosses a saddle connection in ${\mathcal T}_T$. 
We assign the saddle connections in ${\mathcal T}_T$ complexity zero.

Suppose the conclusion of the lemma is false. Then, there exists at least one saddle connection $\tau \subset T$ so
that $[h^{-1}(\tau)]$ contains no saddle connection of $S$. We may choose such a saddle connection $\tau \subset T$ so that it
has minimal complexity with respect to ${\mathcal T}_T$. Since the saddle connections in the triangulation of $T$ came from saddle connections in $S$, 
this minimal complexity must be at least one. We will derive a contradiction from this assumption.

To derive the contradiction, we will prove the following claim. If $\tau$
is any saddle connection in $T$ whose complexity is greater than zero, then there is a convex quadrilateral $Q \subset T$ for which $\tau$ is a diagonal and whose edges are saddle connections with strictly smaller complexity than $\tau$. Here, {\em convex} means that $Q$ is isometric to a convex quadrilateral in $\R^2$. 

We will show how the claim implies the lemma. Let $\tau$ be a saddle connection in $T$ which is a counterexample of
minimal complexity. The claim gives a quadrilateral $Q \subset T$ with diagonal $\tau$.
Let $\nu_1, \ldots, \nu_4 \subset T$ denote the saddle connections which are the edges of $Q$. 
Since $\tau$ was a counterexample with minimal complexity, there are saddle connections 
$\hat \nu_1, \ldots, \hat \nu_4 \subset S$ in the homotopy classes $[h^{-1}(\nu_1)], \ldots, [h^{-1}(\nu_4])]$ respectively. Because
of statement (3), the saddle connections $\hat \nu_1, \ldots, \hat \nu_4$ in $S$ must form a strictly convex quadrilateral $\widehat Q \subset S$. 
The quadrilateral $\widehat Q$ must have diagonals, one of which lies in the homotopy class $[h^{-1}(\tau)]$. See figure \ref{fig:destroyedsaddle}.

\begin{figure}[ht]
\begin{center}
\includegraphics[width=2in]{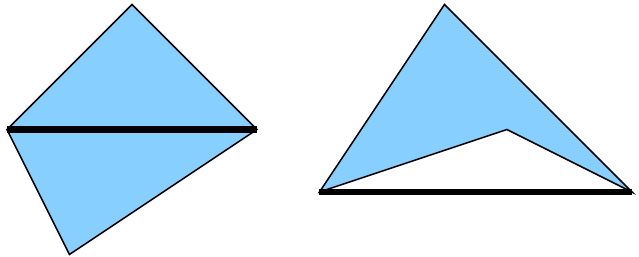}
\caption{To destroy a diagonal of a quadrilateral, the quadrilateral must be made non-convex. This violates property 3 of lemma 
\ref{lem:same_saddles}.}
\label{fig:destroyedsaddle}
\end{center}
\end{figure}

It remains to find $Q$ and prove the claim. Let $\tau \subset T$ be a saddle connection of complexity larger than zero. 
We may assume that $\tau$ is horizontal.
Let $\til \tau$ be a lift to the universal cover, $\til T$. 
Then $\til \tau$ crosses through a sequence of triangles in the triangulation, $\til \Delta_0, \ldots, \til \Delta_K$, which are lifts of triangles in the triangulation ${\mathcal T}_T$.
Consider a developing map $\dev: \til T \to \R^2$ so that $\dev(\til \tau)$
develops to a subset of the line $y=0$. 

As above Proposition \ref{prop:dev}, let $U=\dev(\bigcup_{j=0}^K \til \Delta_j)$.
Let $\til \sigma_1, \ldots, \til \sigma_K$ be the saddle connections between
the triangles, and define $s_j=\dev(\til \sigma_j)$. Then, there is are
endpoints $v_+$ and $v_-$ of one of these saddle connections so that the $y$-coordinates $y_+$ and $y_-$ of $\dev(v_+)$ and $\dev(v_-)$ satisfy 
$[y_-, y_+] = \bigcap_{j=1}^K I_j$. 
These are the vertices of $Q'$ which lie off the horizontal diagonal in Figure \ref{fig:saddleconnection2}.

\begin{figure}[b]
\begin{center}
\includegraphics[width=4.5in]{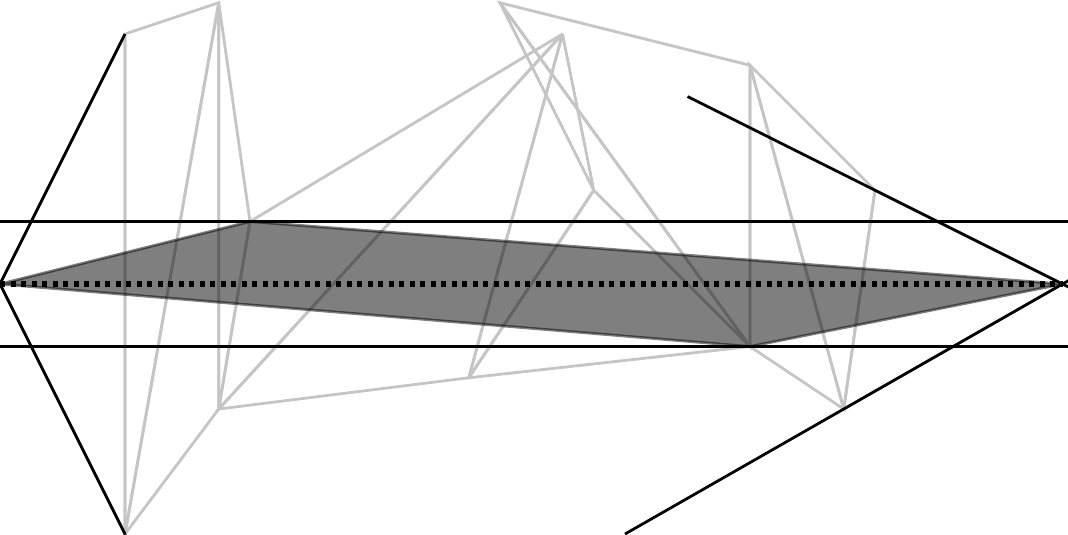}
\caption{The line segment $\dev(\tau)$ is drawn as a horizontal dotted line. The region $V$ is a convex hexagon bounded by black lines, with the gray region $Q'$ inscribed. This the same unfolding
that was shown in Figure \ref{fig:saddleconnection}.}
\label{fig:saddleconnection2}
\end{center}
\end{figure}

Let $V \subset U$ be the convex subset of $\R^2$ as in the proposition above. Let
$$\epsilon:V \to \dev^{-1}(V) \cap \bigcup_{j=0}^K \til \Delta_j$$ 
be the continuous map which is the inverse of the homeomorphism given by restricting $\dev$ as in the proposition. Let $\pi:\til S \to S$ be the universal covering map. Consider quadrilateral 
$Q' \subset V$ with diagonal is $\dev(\tau)$ and whose remaining vertices are given by $\dev(v_+)$ and $\dev(v_-)$. Observe that $Q'$ is inscribed in $V$. Thus, $Q'$ is also convex. Now define $Q=\pi \circ \epsilon(Q')$. This is the desired convex quadrilateral,
with diagonal $\tau$ and additional vertices $\pi(v_+)$ and $\pi(v_-)$. The edges of $Q$
are saddle connections which cross fewer triangles than $\tau$, 
so these saddle connections have smaller complexity. This proves the needed claim.
\end{proof}

The following proposition implies the classification of saddle connections, Theorem \ref{thm:classification_of_saddle_connections}.

\begin{proposition}
\label{prop:lemma_satisfied}
The homeomorphism $h_{c,c'}:S_c \to S_{c'}$ satisfies the conditions of Lemma \ref{lem:same_saddles}.
\end{proposition}
\begin{proof}
It is sufficient to prove that $h_{1,c}$ satisfies the conditions of the lemma, because we can write $h_{c,c'}=h_{1,c}^{-1} \circ h_{1,c'}$. See Proposition \ref{prop:h}. 
Note that if two homeomorphisms satisfy the lemma, then so does their composition. In addition, if $h$ satisfies the lemma, then so does $h^{-1}$ (by the conclusion of the lemma applied to $h$). So, we will restrict to the case of $h_{1,c}$. 

Statement (1) is trivial. We leave it to the reader to triangulate $S_1$. 

Statement (2) follows from propositions \ref{prop:S_1_saddles}. By proposition 
\ref{prop:S_1_saddles} all saddle connections of $S_1$ are the images of saddle connections in the horizontal and slope one directions
under $\G^{\pm}_1$. Observe that that for each saddle connection $\tau$ 
in the horizontal and slope one directions that appears in $S_1$, there is a saddle connection the homotopy class 
$\tau' \in [h_{1,c}(\tau)]$. Let $\sigma$ be an arbitrary saddle connection in $S_1$. Then, $\sigma=\widehat G_1(\tau)$ for some saddle connection $\tau$ of slope zero or one
and some $G \in \G^\pm$ with $\widehat G_1$ denoting the corresponding affine automorphism. 
Let $\tau' \in [h_{1,c}(\tau)]$ be the corresponding saddle connection in $S_c$. Then by Theorem \ref{thm:isotopic_affine_action},
$$\sigma'=\widehat G_c(\tau') \in \widehat G_c([h_{1,c}(\tau)])=[h_{1,c} \circ \widehat G_1(\tau)]=[h_{1,c} (\sigma)]$$
is the desired saddle connection in $S_c$. 

Now we show statement (3) holds.
Let $\sigma$ and $\sigma'$ be saddle connections in the surface $S_1$. 
Let $\theta_0=\{(x,0)~:~x>0\}$ and $\theta_1=\{(x,x)~:~x>0\}$ be horizontal and slope one rays in $\Circ$. 
Then we can choose $\alpha, \alpha' \in \{\theta_0, \theta_1\}$ and $G_1, G'_1 \in \G^\pm_1$ 
such that the holonomies of these saddle connections satisfy $\hol_1(\sigma) \in G_1(\alpha)$ and $\hol_1(\sigma') \in G'_1(\alpha')$.
It follows that the corresponding elements $G_c, G_c' \in \G^\pm_c$ satisfy
$\hol_c \circ h_{1,c}^{-1}(\sigma) \in G_c(\alpha)$ and $\hol_c \circ h_{1,c}^{-1}(\sigma')\in G_c'(\alpha')$. 
We must prove that 
$$\textit{sign} \big( G_1(\alpha) \wedge G_1'(\alpha') \big)=\textit{sign} \big( G_c(\alpha) \wedge G_c'(\alpha') \big),$$
where the sign of the wedge is computed using arbitrary representatives of the classes.
This follows essentially from  Proposition \ref{prop:varphi_c} which defined the the map $\varphi_c: \Circ \to \Circ$.
By this proposition, the statement above is equivalent to 
$$\textit{sign} \big( \varphi_c \circ G_c(\alpha) \wedge \varphi_c \circ G_c'(\alpha')  \big)=\textit{sign} \big( G_c(\alpha) \wedge G_c'(\alpha') \big).$$
Note that for any degree one continuous monotonically increasing map $\psi:\Circ \to \Circ$ which commutes with rotation by $\pi$ satisfies
$$\textit{sign} \big( \psi(\beta) \wedge \psi(\beta') \big) \in \{ 0, \textit{sign} (\beta \wedge \beta')\}$$
for every $\beta, \beta' \in \Circ$. In our setting, we have $G_c(\alpha) \wedge G_c'(\alpha') \neq 0$ if these two directions are fixed by different parabolic subgroups of $\G^\pm_c$.
Note that if the directions $G_1(\alpha)$ and $G_1'(\alpha')$ are distinct, then they are fixed by different parabolic subgroups of $\G^\pm_1$. 
Then, by the commutative diagram in Proposition \ref{prop:varphi_c}, the two directions $\varphi_c \circ G_c(\alpha)$ and $\varphi_c \circ G_c'(\alpha')$ are fixed by distinct parabolics subgroups of $\G^\pm_1$. Therefore $\varphi_c \circ G_c(\alpha) \wedge \varphi_c \circ G_c'(\alpha') \neq 0$. 
\end{proof}

Now we prove Theorem \ref{thm:isotopic_triangulations}, i.e. that the surfaces $S_c$ and $S_{c'}$ admit the same triangulations.

\begin{proof}[Proof of Theorem \ref{thm:isotopic_triangulations}]
Let $\{\sigma_i\}_{i \in \Lambda}$ is a disjoint collection of saddle connections in $S_c$ for $c \geq 1$ which triangulate the surface.
By Proposition \ref{prop:lemma_satisfied}, the homeomorphism $h_{c,c'}$ satisfies the conditions of Lemma \ref{lem:same_saddles}. 
So, we can find saddle connections $\sigma'_i \in [h_{c,c'}(\sigma_i)]$ for all $i$. 
The collection $\{\sigma'_i\}_{i \in \Lambda}$ is also a disjoint collection of saddle connections in $S_{c'}$ which triangulate the surface.
%, because when two saddle connections cross so do every pair of representatives in their homotopy classes. 
%(To see this, take a pair $\sigma$ and $\tau$ of saddle connections which intersect in a translation surface $S$. Consider the universal cover of the punctured
%surface...

We define $h'_{c,c'}:S_c \to S_c'$ to be the homeomorphism which acts affinely on the triangles, and preserves the labeling of edges by $\Lambda$. We claim $h'_{c,c'}$ is isotopic to $h_{c,c'}$.
Because Lemma \ref{lem:same_saddles} is satisfied for the map $h_{c,c'}$ for all pairs of surfaces, we can always do the above construction. Therefore, we can think of $h'_{c,c'}$ as well defined for all $c \geq 1$ and $c' \geq 1$, and this family of homeomorphisms satisfies the conclusions of Proposition \ref{prop:h}. So, to see that $h_{c,c'}$ is isotopic to $h'_{c,c'}$, consider the isotopy given by $h_{c,c''} \circ h'_{c'',c'}$ as $c''$ varies between $c$ and $c'$. 
\end{proof}

We will now begin to study the coding of geodesics in order to prove Theorem \ref{thm:same_geodesics2}. We will again
prove these results by first proving a more general lemma. First we need a definition.

\begin{definition}[$h$-related directions]
Suppose $h:S \to S'$ is a homeomorphism between translation surfaces satisfying the three statements of
Lemma \ref{lem:same_saddles}. Let $\u$ and $\u'$ be unit vectors. We say
$\u$ is $h$-related to $\u'$ if for every saddle connection $\sigma$ on $S$, we have
$$\textit{sign} \big( \hol_S(\sigma) \wedge \u\big) = \textit{sign} \big( \hol_{S'}([h(\sigma)]) \wedge \u'\big).
$$
\end{definition}

Note that in this definition, the homotopy class $[h(\sigma)]$ contains a saddle connection $\sigma' \subset S'$
by the statements in the Lemma. Holonomy is homotopy invariant so we know
$\hol_{S'}([h(\sigma)]) =\hol_{S'}(\sigma').$ 
Therefore, $\u$ is $h$-related to $\u'$ if and only if $\u'$ is $h^{-1}$-related to $\u$.

Recall the definition of interior geodesics, their codes, and coding spaces given near Theorems \ref{thm:same_geodesics} and \ref{thm:same_geodesics2}.

\begin{lemma}
\label{lem:same_geodesics2}
Let $S$ be a translation surface with a collection of saddle connections
$\sT=\{\sigma_i\}_{i \in \Lambda}$ which triangulate the surface. 
Suppose $h:S \to S'$ is a homeomorphism to a translation surface $S'$ which
sends saddle connections in $\sT$ to saddle connections and 
satisfies the three statements of Lemma \ref{lem:same_saddles}. Let $\sT'=\{\sigma'_i=h(\sigma_i)\}_{i \in \Lambda}$ be the image saddle connections
which triangulate $S'$. Suppose that the unit vector $\u$ is $h$-related to the unit vector $\u'$.
Then the coding spaces $\Omega_{\u}, \Omega'_{\u'} \subset \Lambda^\Z$ which code geodesics 
in directions $\u$ and $\u'$ using these triangulations on $S$ and $S'$, respectively, are equal.
\end{lemma}

\begin{remark}
The condition that each $h(\sigma_i)$ be a saddle connection may be superfluous. The author is not sure if any homeomorphism satisfying Lemma \ref{lem:same_saddles}
is isotopic to a homeomorphism satisfying this condition.
\end{remark}

\begin{proof}
We will show that $\Omega_{\u} \subset \Omega'_{\u'}$. This suffices to prove the lemma,
because the conclusion of Lemma \ref{lem:same_saddles} implies that the statements of Lemma \ref{lem:same_saddles} are also
satisfied by $h^{-1}$. So the same argument will yield $\Omega'_{\u'} \subset \Omega_{\u}$.

%By possibly rotating the surfaces $S$ and $S'$ we can assume that the unit vectors $\u$ and $\u'$ are both given by $(1,0)$.

To show that $\Omega_\u \subset \Omega'_{\u'}$, it suffices to show that every finite word $w$ which appears in $\Omega_u$ also appears in $\Omega'_{\u'}$.
We will actually prove a stronger statement by induction in the length of $w$. 
To give this statement, we need to describe some geometry.

We may assume by rotating the two surfaces that the unit vectors $\u$ and $\u'$
are both $(1,0)$. 
Let $w=w_1 \ldots w_K$ be a finite word which appears in $\Omega_u$. 
Then there is an interior geodesic $\gamma$ on $S$ in direction $\u=(1,0)$
whose code contains $w$ as a subword. Thus, we may choose an interval $[a,b] \subset \R$ so that $\gamma(a)$ 
and $\gamma(b)$ lie in the interior of triangles and $\gamma(t)$ passes over the sequence of saddle connections
$\sigma_{w_j}$ for $j=1,\ldots, K$ as $t$ runs from $a$ to $b$. 
We orient these saddle connections upward.
Let $\Delta_0,\ldots, \Delta_K$ be the sequence
of triangles passed through in this time. These definitions satisfy the considerations laid out above Proposition \ref{prop:dev},
i.e. the sequence is non-backtracking and the each saddle connection lies between adjacent triangles.

As above the proposition, we lift this picture to the universal cover $\til S$. 
We obtain a lifted geodesic $\til \gamma$,
lifts $\til \sigma_{1}, \ldots, \til \sigma_K$ of $\sigma_{w_1},\ldots,\sigma_{w_K}$ respectively, and lifts $\til \Delta_0,\ldots, \til \Delta_K$ of $\Delta_0,\ldots, \Delta_K$.
Consider the developing map $\dev: \til S \to \R^2$ 
defined so that $\dev \circ \til \gamma(t)=(t,0)$. Let $s_j=\dev(\til \sigma_j)$.
This is always a segment in the plane with one vertex with positive $y$-coordinate
and one vertex with negative $y$-coordinate. We call the former {\em the top vertex $T_j \in \R^2$} and the later {\em the bottom vertex $B_j \in \R^2$.}
We let $t_j$ and $b_j$ denote the endpoints of $\til \sigma_j$ for which 
$\dev(t_j)=T_j$ and $\dev(b_j)=B_j$. 
We define $I_j$ to be the closed interval of $y$-coordinates contained in $\dev(\sigma_j)$
as in equation \ref{eq:I}.
Observe that $\bigcap_{j=1}^K I_j$ contains zero in its interior because the unfolding contains 
$\dev \circ \til \gamma([a,b])$, a segment with $y$-coordinate zero.  
We say $T_j$ is a {\em leading top vertex}
if its $y$-coordinate is non-strictly smaller than all other $y$-coordinates of
top vertices. Similarly, we say $B_j$ is a {\em leading bottom vertex} if
its $y$-coordinate is non-strictly larger than all other $y$-coordinates of
bottom vertices. 

Now consider the surface $S'$. We have assumed that $h$ sends the saddle connections in the triangulation
of $S$ to saddle connections in $S'$. Let
$\gamma'=h \circ \gamma$. Then this (non-geodesic) curve in $S'$ crosses saddle connections in the sequence given by $\omega$ as well. Choose a lift $\til \gamma'$
of $\gamma'$ to the universal cover $\til S'$. We may choose a lift of the homeomorphism $h$ to a homeomorphism $\til h:\til S \to \til S'$ so that 
$\til \gamma'=\til h \circ \til \gamma$. Let $\til \sigma'_j=\til h(\til \sigma_j)$ and $\til \Delta'_j=\til h(\til \Delta_j)$. This is a sequence of triangles and saddle connections crossed by $\til \gamma'$ in $\til S'$. Let $t'_j=\til h(t_j)$ and $b'_j=\til h(b_j)$ satisfying the considerations laid out above Proposition \ref{prop:dev}. Choose a developing map 
$\dev':\til S' \to \R^2$ and define $T'_j=\dev'(t_j)$ and $B'_j=\dev'(b_j)$ and $s'_j=\dev'(\til \sigma'_j)$.
We again call the points $T'_j$ {\em top vertices} and call $T'_j$ a {\em leading top vertex} if its $y$-coordinate is non-strictly smaller than all other top vertices in the developing map for $S'$. We make the analogous definitions for bottom vertices. The saddle connections $\til \sigma'_j$ inherit an orientation from $\til \sigma_j$.
This is again the upward orientation, because we assumed $\u$ was $h$-related to $\u'$, and normalized so both these vectors were the vector $(1,0)$. In particular each $T_j$ has greater $y$-coordinate than each $B_j$. 

We remark that the homeomorphism $h$ must be orientation preserving because of statement (3) of Lemma \ref{lem:same_saddles} applied to the saddle connections in the triangulation $\sT$. The saddle connections
$\til \sigma'_j$ have an upward orientation, so the curve $\gamma'$ must move from the left side of each such saddle connection to the right side. Therefore, the sequence of developed triangles
$\dev(\Delta'_j)$ moves rightward as $j$ increases.

Our lemma follows from the following claim. Let $w$ be any word which appears in $\Omega_\u$. Let $K>0$
be the the length of $w$. Then the intersection $\bigcap_{j=1}^K I'_j$ is a non-degenerate closed interval.
Moreover, the set of indices of
leading top vertices in $S$ is the same as the set of indices of top vertices in $S'$ and the set of indices of leading bottom vertices in $S$ is the same as the set of indices of bottom vertices in $S'$.

First we will prove this statement for words of length one. Suppose $w=w_1$
is a word appearing in $\Omega_u$. Then $I_1$ is non-degenerate because $\til \sigma'_1$ has
an upward orientation as mentioned above. The collection of indices of leading top and bottom
vertices is necessarily just $\{1\}$ for both $S$ and $S'$ as there is only one saddle connection to consider.

Now assume the claim is true for all words of length $K$. We will prove it holds for words of length $K+1$.
Any such word can be written $w\lambda$ where $w$ is a word of length $K$, and $\lambda \in \Lambda$ is an additional letter. Then the edge $\til \sigma_{K+1}$ is a lift of $\sigma_{\lambda}$. The developed segment
$s_{K+1}$ shares one endpoint with the prior segment $s_K$, either a top vertex or a bottom vertex. By possibly reflecting in the horizontal, we may assume without loss of generality that the common segment is a bottom vertex, i.e. $B_{K}=B_{K+1}$. 
We consider three possible configurations for the top vertex $T_{K+1}$
(recalling that the $y$-coordinate of $T_{K+1}$ is larger than that of the bottom leader):
\begin{enumerate}
\item The vertex $T_{K+1}$ is the only leader for the word $w \lambda$. In this case, its $y$-coordinate is smaller than the $y$-coordinate of the leading top vertex for the subword $w$.
\item The vertex $T_{K+1}$ is a new top leader, sharing the role with some other top leaders. In this case, the $y$-coordinate of $T_{K+1}$ equals the $y$-coordinate of the leading top vertex (or vertices) for the subword $w$.
\item The top vertex $T_{K+1}$ is not a leader. That is, the $y$-coordinate of $T_{K+1}$ is larger than that of the leading top vertex for the subword $w$.
\end{enumerate}
We will prove the claim by analyzing each of these cases. We establish some common notation, before addressing these cases. Let $U$ and $U'$ be the unfoldings of the word $w$ (rather than $w\lambda$) on the surfaces $S$ and $S'$,
respectively. Let $L \in \{1,\ldots,K\}$ denote the largest index so that $T_L$
is a top leader of the unfolding $U$ of the word $w$ on $S$. By inductive hypothesis, $L$ is also the largest index so that $T'_L$ is a top leader of the unfolding $U'$ of the word $w$ on $S'$. Similarly, let $\ell\in \{1,\ldots,K\}$ be the largest index so that $B_\ell$ is a bottom leader. 
Let $V \subset U$ and $V' \subset U'$ be the convex regions guaranteed to exist by 
Proposition \ref{prop:dev}. Here to get $V'$ we are using the inductive hypothesis to ensure that the unfolding of the word $w$ satisfies the conditions of Proposition \ref{prop:dev}. Let $\epsilon$ denote the inverses of the homeomorphism obtained using the Proposition by restricting $\dev$ to $\dev^{-1}(V) \cap \bigcup_{i=0}^K \Delta_K$. 
Note that the segments $s_{K+1}=\overline{B_{K+1}T_{K+1}}$ and $s'_{K+1}=\overline{B'_{K+1}T'_{K+1}}$ appear in the boundary of $U$ and $U'$ respectively, as the saddle connections $\til \sigma_{K+1}$ and $\til \sigma'_{K+1}$ are edges of $\til \Delta_{K}$ and $\til \Delta'_{K}$. 

Consider case (1). In this case, the $y$-coordinate of $T_{K+1}$ lies strictly between the $y$-coordinates of $B_\ell$ and $T_L$. So, $T_{K+1} \in V$. Consider the segments 
$\overline{B_\ell T_{K+1}}$ and $\overline{T_L T_{K+1}}$. 
These are chords of the convex set $V$, and the first
is oriented upward and the second is oriented downward. 
See Figure \ref{fig:case1}. Then 
$\tau_B=\epsilon(B_\ell T_{K+1})$ and $\tau_T=\epsilon(\overline{T_L T_{K+1}})$
are saddle connections on $\til S$. Let $\tau'_B$ and $\tau'_T$ be the saddle connections in the homotopy class of $[\til h(\tau_B)]$ and $[\til h(\tau_T)]$
given by the assumption that the surfaces satisfy Lemma \ref{lem:same_saddles}.
The fact that $\u$ is $h$-related to $\u'$ implies that $\tau'_B$ 
is oriented upward while $\tau'_T$ is oriented downward. Consider the developments
$\dev'(\tau'_B)=\overline{B'_\ell T'_{K+1}}$ and $\dev'(\tau'_T)=\overline{T'_L T'_{K+1}}$. Because these segments have upward and downward orientations respectively, we know that the $y$-coordinate of $T'_{K+1}$ lies strictly between
the $y$-coordinates of $B'_\ell$ and $T'_L$. Thus $\bigcap_{j=1}^{K+1} I_j$ is the interval from the $y$-coordinate of $B'_\ell$ to the $y$-coordinate of $T'_{K+1}$,
and this is non-degenerate. In addition, $T_{K+1}$ will be the only new top leader for unfolding of the word $w\lambda$ on the surface $S'$. This verifies the claim in this case. 

\begin{figure}[t]
\begin{center}
\includegraphics[width=4in]{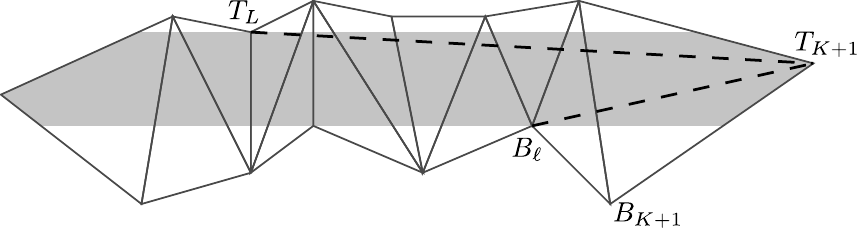}
\caption{An example unfolding $U$ in case (1) drawn with gray subset $V$ together with
segments $\overline{B_\ell T_{K+1}}$ and $\overline{T_L T_{K+1}}$.}
\label{fig:case1}
\end{center}
\end{figure}

In case (2), the argument is very similar. Here, we find that the segment
$\overline{T_L T_{K+1}}$ is horizontal and contains no other vertices.
We define saddle connections $\tau_T=\epsilon(\overline{T_L T_{K+1}})$ and $\tau'_T$
as above. By the same argument, we see that $\tau'_T$ must be horizontal. It follows
that $T'_{K+1}$ is a new leader, sharing the role with the prior leader or leaders
including $T'_L$. 

Case (3) is a bit more subtle. In this case, $y$-coordinate of $T_{K+1}$ is larger than that of the leading top vertex for the subword $w$. Consider the curve $\alpha$ in the unfolding joining $T_L$ to $T_{K+1}$ which runs from $T_L$ rightward to the triangle $\dev(\til \Delta_K)$ inside the set $V$ and then moves to $T_{K+1}$ while staying inside of $\dev(\til \Delta_K)$. We can lift $\alpha$ to a curve $\til \alpha$ in $\til S$ using $\epsilon$ and the identification between $\Delta_K$ and $\dev(\Delta_K)$. Let
$[\til \alpha]$ denote the homotopy class of all curves from $t_L$ to $t_{k+1}$
including $\til \alpha$. Recalling that $\til S$ was the universal cover branched over the vertices, the minimal length of a curve in this class is realized by some
curve $\beta$ which is a sequence of one or more upward oriented saddle connections contained in $\bigcup_{j=0}^K \til \Delta_K$. See Figure \ref{fig:case3}. Because $\u$ was $h$-related to
$\u'$ and because $h$ satisfies Lemma \ref{lem:same_saddles}, we know that we can straighten $\til h(\beta)$ to a chain of upward oriented
saddle connections $\beta'$ joining $t'_L$ to $t'_{K+1}$. We conclude that the $y$-coordinate of $T'_{K+1}$ is larger than the $y$-coordinate of $T'_L$ as desired.
\end{proof}

\begin{figure}[t]
\begin{center}
\includegraphics[width=4in]{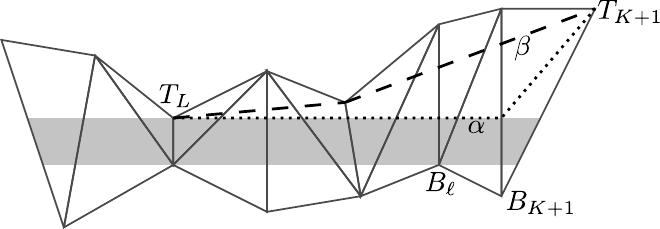}
\caption{An example unfolding $U$ in case (3) drawn with gray subset $V$ together with segments $\overline{B_\ell T_{K+1}}$ and $\overline{T_L T_{K+1}}$.}
\label{fig:case3}
\end{center}
\end{figure}

\begin{proof}[Proof of Theorem \ref{thm:same_geodesics2}]
Let $\varphi_c:\Circ \to \Circ$ be the continuous monotonic map of degree one constructed in Proposition \ref{prop:varphi_c}. This map restricts to a bijection on the directions containing saddle connections by Theorem \ref{thm:classification_of_saddle_connections}. It follows from this
and non-strict monotonicity that for each $\u \in \Circ$, we have that $\u'=\varphi_c(\u)$ is $h_{1,c}$-related to $\u$. The conclusion follows
by applying Lemma \ref{lem:same_geodesics2}. We can apply this lemma, because
Proposition \ref{prop:lemma_satisfied} states that the homeomorphism $h_{1,c}$ satisfies Lemma \ref{lem:same_saddles}, and Theorem \ref{thm:isotopic_triangulations},
guarantees that we can replace $h_{1,c}$ by an isotopic homeomorphism which sends saddle connections in the triangulation to saddle connections.
\end{proof}

\section{No other affine automorphisms}

The last step to the proof of Theorems \ref{thm:veech_groups} and \ref{thm:isotopic_affine_action}
is to demonstrate that all affine automorphisms of the surface lie in the group generated by the elements we listed.

\begin{lemma}
All affine automorphisms of the surface $S_c$ are contained in the group 
generated by $\widehat{-I}_c$, $\hat{A}_c$, $\hat{D}_c$, and $\hat{E}_c$.
\end{lemma}
\begin{proof}
Let us suppose that for some $c \geq 1$ there is an $M \in \GL(2, \R)$ in the Veech group $\Gamma(S_c)$ and
a corresponding element $\hat{M}$ in the affine automorphism group  $\Aff(S_c)$. We will prove that
$\hat{M}$ lies in the group generated by the four elements $\widehat{-I}_c$, $\hat{A}_c$, $\hat{D}_c$, and 
$\hat{E}_c$.

Let $\theta=\{(x,0)~:~x>0\} \in \Circ$ be the horizontal direction.
We know that the image $M(\theta)$ must contain the holonomies of saddle connections of $S_c$. Further more
the horizontal and slope one directions can be distinguished, since the smallest area maximal cylinder in the horizontal
direction has two cone singularities in its boundary, while the smallest area maximal cylinder in the slope one direction
has four cone singularities in its boundary. Thus, by theorem \ref{thm:classification_of_saddle_connections}, there must be an
element $N_c \in \G^\pm_c$ satisfying $M(\theta) =N_c(\theta)$. It follows that $N_c^{-1} \circ M$ preserves the 
horizontal direction. 

There must be a corresponding element element $\hat{N}_c^{-1} \circ \hat{M} \in \Aff(S_c)$ with derivative
$N_c^{-1} \circ M$. The automorphism must fix the decomposition into horizontal cylinders, and fix each cylinder
in the decomposition (because the cylinders have distinct areas). The smallest area horizontal cylinder is
isometric in each $S_c$. It is built from two triangles, the convex hull of $(0,0)$, $(1,1)$, and $(-1,1)$ and
the same triangle rotated by $\pi$, with diagonal sides of the first glued to the diagonal sides of the second 
by translation. $\hat{N}_c^{-1} \circ \hat{M} \in \Aff(S_c)$ must preserve this cylinder and permute the pair of cone singularities in the boundary. Therefore
$$N_c^{-1} \circ M=\left[\begin{array}{cc} 1 & 2 n \\ 0 & 1 \end{array}\right]=D_c^n
\quad \textrm{or} \quad
N_c^{-1} \circ M=\left[\begin{array}{cc} 1 & -2 n \\ 0 & -1 \end{array}\right]=-I \circ A_c \circ D_c^n$$
for some $n \in \Z$. Therefore, $M=N_c \circ D_c^n$ or $M=N_c \circ -I \circ A_c \circ D_c^n$, all of which
lie in $\G^\pm_c$. Therefore, by Proposition \ref{prop:bijection}, the corresponding affine automorphism
satisfies $\hat M=\hat N_c \circ \hat D_c^n$ or $\hat M=\hat N_c \circ \widehat{-I} \circ \hat A_c \circ \hat D_c^n$.
\end{proof}

\noindent{\bf Acknowledgements.} The author would like to thank the referee for many helpful comments and criticisms. He would also like to thank Samuel Leli\`evre for helpful comments involving the proof of Lemma \ref{lem:same_saddles}.

\bibliographystyle{amsalpha}
\bibliography{../bibliography}

\newcommand{\cn}[1]{{\bf [#1]:}} \def\scr{\mathcal} \def\cprime{$'$}
\providecommand{\bysame}{\leavevmode\hbox to3em{\hrulefill}\thinspace}
\providecommand{\MR}{\relax\ifhmode\unskip\space\fi MR }
\providecommand{\MRhref}[2]{%
  \href{http://www.ams.org/mathscinet-getitem?mr=#1}{#2}
}
\providecommand{\href}[2]{#2}
\begin{thebibliography}{ANSS02}

\bibitem[ANSS02]{ANSS02}
Jon Aaronson, Hitoshi Nakada, Omri Sarig, and Rita Solomyak, \emph{Invariant
  measures and asymptotics for some skew products}, Israel J. Math.
  \textbf{128} (2002), 93--134. \MR{1910377 (2003f:37012)}

\bibitem[FU11]{FUpreprint}
Krzysztof Fr{\c{a}}czek and Corinna Ulcigrai, \emph{Non-ergodic z-periodic
  billiards and infinite translation surfaces}, preprint
  \url{http://arxiv.org/abs/1109.4584}, 2011.

\bibitem[Ghy87]{Ghys87}
Etienne Ghys, \emph{{Groupes d'hom\'eomorphismes du cercle et cohomologie
  born\'ee. (Homeomorphism groups of the circle and bounded cohomology)}},
  {Differential equations, Proc. Lefschetz Centen. Conf., Mexico City/Mex.
  1984, Pt. III, Contemp. Math. 58.3, 81-106}, 1987.

\bibitem[HHW10]{HHW10}
W.~Patrick Hooper, Pascal Hubert, and Barak Weiss, \emph{Dynamics on the
  infinite staircase}, To appear in Discrete and Continuous Dynamical Systems -
  Series A, 2010.

\bibitem[Hoo08]{Higl}
W.~Patrick Hooper, \emph{Dynamics on an infinite surface with the lattice
  property}, preprint, \url{http://arxiv.org/abs/0802.0189}, 2008.

\bibitem[Hoo10]{Hinf}
W.~Patrick Hooper, \emph{The invariant measures of some infinite interval
  exchange maps}, preprint, \url{http://arxiv.org/abs/1005.1902}, 2010.

\bibitem[Hoo12]{Higl2}
W.~Patrick Hooper, \emph{An infinite surface with the lattice property {II}:
  {D}ynamics of pseudo-{A}nosovs}, in preparation, 2012.

\bibitem[HS10]{HS09}
Pascal Hubert and Gabriela Schmith{\"u}sen, \emph{Infinite translation surfaces
  with infinitely generated {V}eech groups}, 2010. \MR{2753950 (2012e:37075)}

\bibitem[HW10]{HW10}
W.~Patrick Hooper and Barak Weiss, \emph{Generalized staircases: recurrence and
  symmetry}, Ann. Inst. Fourier (2010), to appear.

\bibitem[HW12]{HWarxiv12}
Pascal Hubert and Barak Weiss, \emph{Ergodicity for infinite periodic
  translation surfaces}, 2012.

\bibitem[MT98]{MT98}
Katsuhiko Matsuzaki and Masahiko Taniguchi, \emph{Hyperbolic manifolds and
  {K}leinian groups}, Oxford Mathematical Monographs, The Clarendon Press
  Oxford University Press, New York, 1998, Oxford Science Publications.
  \MR{1638795 (99g:30055)}

\bibitem[RT11]{RTarxiv11}
David Ralston and Serge Troubetzkoy, \emph{Ergodicity of certain cocycles over
  certain interval exchanges}, 2011, preprint,
  \url{http://arxiv.org/abs/1111.2465}.

\bibitem[RT12]{RTarxiv12}
\bysame, \emph{Ergodic infinite group extensions of geodesic flows on
  translation surfaces}, 2012, preprint, \url{http://arxiv.org/abs/1201.3738}.

\bibitem[Thu97]{Thurston}
William~P. Thurston, \emph{Three-dimensional geometry and topology. {V}ol. 1},
  Princeton Mathematical Series, vol.~35, Princeton University Press,
  Princeton, NJ, 1997, Edited by Silvio Levy. \MR{1435975 (97m:57016)}

\bibitem[Vee89]{V}
W.~A. Veech, \emph{Teichm\"uller curves in moduli space, {E}isenstein series
  and an application to triangular billiards}, Invent. Math. \textbf{97}
  (1989), no.~3, 553--583. \MR{1005006 (91h:58083a)}

\bibitem[Yoc10]{Y10}
Jean-Christophe Yoccoz, \emph{{Interval exchange maps and translation
  surfaces.}}, {Einsiedler, Manfred Leopold (ed.) et al., Homogeneous flows,
  moduli spaces and arithmetic. Proceedings of the Clay Mathematics Institute
  summer school, Centro di Recerca Mathematica Ennio De Giorgi, Pisa, Italy,
  June 11--July 6, 2007. Providence, RI: American Mathematical Society (AMS);
  Cambridge, MA: Clay Mathematics Institute. Clay Mathematics Proceedings 10,
  1-69 (2010).}, 2010.

\bibitem[Zor06]{Zorich06}
Anton Zorich, \emph{Flat surfaces}, Frontiers in number theory, physics, and
  geometry. I, Springer, Berlin, 2006, pp.~437--583. \MR{2261104 (2007i:37070)}

\end{thebibliography}
\end{document}